\documentclass[a4paper,11pt,reqno]{amsart}
\usepackage{lmodern}

\usepackage{amssymb}
\usepackage{siunitx}
\usepackage{epsfig}
\usepackage{amsfonts,amsrefs}
\usepackage{amsmath}
\usepackage{euscript}
\usepackage{amscd}
\usepackage{amsthm}
\usepackage{enumitem}
\DeclareMathAlphabet{\mathpzc}{OT1}{pzc}{m}{it}
\usepackage{enumitem}
\usepackage{color}
\usepackage{mathtools}
\usepackage[hypertexnames=false, colorlinks, citecolor=red, linkcolor=blue, urlcolor=red]{hyperref}
\usepackage[utf8]{inputenc}
\usepackage[T1]{fontenc} 
\usepackage{marginnote}
\usepackage{marvosym}

\usepackage[normalem]{ulem}

\newcommand{\marginextend}[1]{ \addtolength{\oddsidemargin}{-#1}  \addtolength{\evensidemargin}{-#1}
	\addtolength{\textwidth}{#1}\addtolength{\textwidth}{#1}}
\newcommand{\updownextend}[1]{ \addtolength{\topmargin}{-#1}  \addtolength{\textheight}{#1}
	\addtolength{\textheight}{#1}}
\marginextend{1.5cm}
\updownextend{0cm}
\allowdisplaybreaks[4]

\usepackage{pst-node}
\usepackage{tikz-cd}
\usepackage{mathrsfs}
\usepackage[most]{tcolorbox}

\DeclareFontFamily{OT1}{pzc}{}
\DeclareFontShape{OT1}{pzc}{m}{it}{<-> s * [1.10] pzcmi7t}{}
\DeclareMathAlphabet{\mathpzc}{OT1}{pzc}{m}{it}

\DeclareSymbolFont{SY}{U}{psy}{m}{n}
\DeclareMathSymbol{\emptyset}{\mathord}{SY}{'306}

\theoremstyle{plain}

\newtheorem{thm}{Theorem}[section]
\newtheorem*{thm*}{Theorem}
\newtheorem{cor}[thm]{Corollary}
\newtheorem{lem}[thm]{Lemma}
\newtheorem{prop}[thm]{Proposition}
\newtheorem{defn}[thm]{Definition}
\newtheorem{rem}[thm]{Remark}

\newtheoremstyle{mainthmstyle}%
  {}{}
  {\itshape}
  {}
  {\bfseries}
  {}
  {0pt}
  {\thmname{#1}. \thmnote{#3}} 

\theoremstyle{mainthmstyle}

\newtheoremstyle{named}{}{}{\itshape}{}{\bfseries}{.}{.5em}{#1 \thmnote{#3}}
\theoremstyle{named}

\tcolorboxenvironment{maintheorem}{
  colback=white,
  colframe=black,
  boxrule=0.8pt
}

\numberwithin{equation}{section}

\def\beq{\begin{eqnarray}}
	\def\eeq{\end{eqnarray}}
\def\beqa{\begin{eqnarray*}}
	\def\eeqa{\end{eqnarray*}}

\newcommand{\be}{\begin{equation}}
	\newcommand{\ee}{\end{equation}}
\newcommand{\bea}{\begin{eqnarray}}
	\newcommand{\eea}{\end{eqnarray}}
\newcommand{\Bea}{\begin{eqnarray*}}
	\newcommand{\Eea}{\end{eqnarray*}}

\newcounter{cnt1}
\newcounter{cnt2}
\newcounter{cnt3}
\newcommand{\blr}{\begin{list}{$($\roman{cnt1}$)$}
		{\usecounter{cnt1} \setlength{\topsep}{0pt}
			\setlength{\itemsep}{0pt}}}
	\newcommand{\bla}{\begin{list}{$($\alph{cnt2}$)$}
			{\usecounter{cnt2} \setlength{\topsep}{0pt}
				\setlength{\itemsep}{0pt}}}
		\newcommand{\bln}{\begin{list}{$($\arabic{cnt3}$)$}
				{\usecounter{cnt3} \setlength{\topsep}{0pt}
					\setlength{\itemsep}{0pt}}}
			\newcommand{\el}{\end{list}}
		
\title[Continuous family of compact quantum metric spaces]{Continuous family of compact quantum metric space structures from cocycle twisted crossed product $\textrm{C}^{\ast}$-algebras}

\author[A. Chattopadhyay]{Arnab Chattopadhyay}
\author[S. Joardar]{Soumalya Joardar}

\address[A. Chattopadhyay]{Indian Institute of Science Education And Research Kolkata, Mohanpur 741246, Nadia, West Bengal, India} \email{ac23rs002@iiserkol.ac.in}

\address[S. Joardar]{Indian Institute of Science Education And Research Kolkata, Mohanpur 741246, Nadia, West Bengal, India} \email{soumalya@iiserkol.ac.in}

\begin{document}

\begin{abstract}
    We establish the existence of a three-parameter family of compact quantum metric space structures on cocycle twisted crossed products by discrete groups. We are mainly interested in the case where the acting group has exponential/subexponential growth. We prove that the family is jointly continuous with respect to the parameters when the acting group is exact. We obtain quantitative upper and lower bounds for the associated metric dimensions. In particular, the bounds are helpful to prove the failure of lower semicontinuity of the metric dimension with respect to the quantum Gromov-Hausdorff distance. We also prove invariance of metric dimension under zero quantum Gromov-Hausdorff distance. 
\end{abstract}
\maketitle
\section{Introduction}
Compact quantum metric spaces (CQMS for short), introduced by M. Rieffel, provide a framework in which geometric ideas such as distance, dimension, and convergence can be extended from classical compact metric spaces to operator algebras.  Over the last two decades, this theory has become an important component of noncommutative geometry, particularly in the study of deformation phenomena, approximation by finite-dimensional structures, and quantum analogues of classical metric invariants. M. Rieffel initially proposed compact quantum metric spaces in the framework of order unit spaces (see \cites{Rieffel-Metric, Rieffel-Gromov, Rieffel-Sphere}). Later it was extended to the set up of $C^{\ast}$-algebras in a series of papers by F. Latr\'emoli\`ere, H. Li et al (\cites{Latre-Prop1, Latre-Prop2, Latre-Prop3}; \cites{Hanfeng-CQMS, Hanfeng-qGH}). Subsequently, it was extended for complete operator systems by J. Kaad and D. Kayed in \cite{Kaad-Kyed}.

A recurring theme in mathematical physics is that geometric structures often persist under deformation. Classical examples include deformation quantization, noncommutative tori, twisted group algebras, and more general cocycle deformations of crossed products. Such deformations frequently preserve substantial algebraic information while modifying the underlying geometry in a highly nontrivial way. An important question is therefore to understand how quantum metric structures behave under these deformations and to identify geometric quantities that remain stable, vary continuously, or exhibit phase transitions.

The present article is concerned with cocycle twisted crossed product $C^{\ast}$-algebras associated with actions of discrete groups. These algebras simultaneously encode the geometry of a coefficient algebra, the dynamics of a group action, and the deformation data carried by a unitary $2$-cocycle. They form a large class of examples encompassing twisted group $C^{\ast}$-algebras, noncommutative tori, and many other deformation models arising in operator algebras and mathematical physics. Despite their importance, a systematic quantum metric treatment of cocycle twisted crossed products remains largely undeveloped.

The first objective of this paper is to construct compact quantum metric space structures on such twisted crossed products. Starting from a compact quantum metric space $(A,L_A)$ and a discrete group equipped with a proper length function, we introduce a family of weighted Lip-norms on the twisted crossed product $A\rtimes_{r,\rho,\sigma}\Gamma$. The construction combines the metric information coming from the coefficient algebra with a weighted control of Fourier coefficients in the group direction. The weights depend on two parameters and reflect the interaction between the decay imposed by the Lip-norm and the growth of the underlying group. 

Our motivation for introducing these weighted structures are two-fold. The first motivating factor is the observation that the growth properties of groups should influence the geometry of the corresponding noncommutative spaces. In this context, note that compact quantum metric space structures on crossed product $C^{\ast}$-algebras have been studied by A. Hawkins in \cite{Hawkins-Crossed-CQMS} and by M. Klisse in \cite{Klisse-Crossed-CQMS}. In those articles, the main idea was to get CQMS structures from suitable spectral triples on the crossed products. But in those approaches, although the length function of the acting groups played a role, the large scale geometry of the group didn't play any prominent role. More recently, CQMS structures on crossed product by groups of polynomial growth has also been studied by A. Austad in \cite{Austad-Crossed-CQMS}. But CQMS structures on crossed product by exponential/subexponential growth group is largely missing from the literature. In classical geometry, growth controls the asymptotic size of metric balls and is closely related to dimensional phenomena. In the noncommutative setting, one expects analogous behavior to emerge through covering invariants associated with quantum metric spaces. The second motivation is that for a large class of transformation groupoids, one necessarily loses the rapid decay property of the length function so that the techniques in \cites{Anton-Chris-CQMS, Soumalya-Arnab2} do not work anymore. Note that the weighted Lip-norms fail to satisfy the Leibnitz-type property. Therefore, one has to abandon the quantum propinquity of Latr\'emoli\`ere. Instead, we stick to the notion of quantum Gromov-Hausdorff distance in complete operator systems in the sense of \cite{Kaad-Kyed}. In \cite{Soumalya-Arnab2}, it was observed that if the discrete group has the exponential growth property, any Lip-norm satisfying the Leibnitz property forces the metric dimension to be $+\infty$. Therefore, from the point of view of metric dimension, if one wants to handle $C^{\ast}$-dynamical system coming from groups with exponential/ subexponential growth, then  the loss of the Leibnitz property is quite natural.

The second objective of the paper is to study the metric dimension of the resulting compact quantum metric spaces. Metric dimension was introduced by D. Kerr as a noncommutative covering invariant and has subsequently appeared in several contexts related to entropy and dimension theory. Our results reveal a striking threshold phenomenon governed by the competition between the growth rate of the acting group and the decay rate built into the Lip-norm. When decay dominates growth, the group contribution becomes invisible from the perspective of the metric dimension. However, at a critical scaling, the geometry retains a nontrivial contribution from the group. Thus the family of quantum metric spaces constructed here exhibits behaviour analogous to a phase transition, with the growth exponent of the group playing the role of a critical parameter. These bounds also provide examples of crossed product $C^{\ast}$-algebras with finite metric dimension where the acting group has exponential/subexponential growth.

A further theme of this work is deformation stability. Given a continuous family of unitary $2$-cocycles, one obtains a corresponding family of twisted crossed products. From the perspective of noncommutative geometry, it is natural to ask whether the associated quantum metric structures vary continuously under such cocycle deformations. Questions of this type have played a central role in the study of noncommutative tori and related deformation families. We show that the weighted compact quantum metric space structures introduced in this article vary continuously both with respect to the cocycle parameter and with respect to the scaling parameters appearing in the definition of the Lip-norm jointly (Theorem \ref{mainthm}). Consequently, the resulting family forms a continuous quantum-geometric deformation of the underlying crossed product. We provide concrete examples (Section \ref{Examples}) of such phenomenon. As one of the main motivations of this paper is to deal with the case of exponential/subexponential growth groups, we provide examples which arise from the $C^{\ast}$-dynamical systems involving hyperbolic surface groups. Let $\Gamma=\pi_{1}(\Sigma_g))$ be the fundamental group of a closed orientable surface of genus $g\geq 2$. Since $H^{2}(\Gamma,\mathbb{R})\cong \mathbb{R}$, every real-valued $2$-cocycle gives rise to a one-parameter family of unitary cocycles $\{\sigma_{\theta}\}_{\theta\in\mathbb{R}}$, producing a continuous deformation of the associated crossed product algebras. Our continuity results apply in particular to two natural classes of isometric actions of $\Gamma$. The first is the action on its profinite completion $\widehat{\Gamma}$, equipped with the standard ultrametric arising from a descending chain of finite-index normal subgroups. The second is an isometric action on the compact Lie group $SU(2)$. In both cases, the cocycle parameter $\theta$ generates a continuous family of twisted crossed products, and our results show that the associated compact quantum metric spaces vary continuously in the quantum Gromov--Hausdorff sense. These examples provide a geometric realization of our general deformation theory and illustrate how cocycle deformations interact with actions on compact spaces of markedly different character: a totally disconnected profinite space in the first case and a connected compact Lie group in the second. 

In general, the interaction between the metric dimension of compact quantum metric spaces and the quantum Gromov-Hausdorff distance is not very well understood. It is known that the metric dimension fails to be lower semicontinuous with respect to the quantum Gromov-Hausdorff distance. This can be seen in the convergence of matrix algebras to the sphere (\cite{Rieffel-Sphere}). In this example, matrix algebras, being finite dimensional vector spaces, have zero metric dimensions. Yet they converge to the sphere which has positive metric dimension. But this might as well be attributed to the finite dimensionality of the algebras instead of a true metric geometric phenomenon. In this paper, we have examples of compact quantum metric spaces of constant metric dimension converging to a CQMS of metric dimension strictly greater than the constant. Moreover, all the underlying spaces are infinite dimensional vector spaces (Remark \ref{lowersemicontinuity}). This illustrates the failure of lower semicontinuity of the metric dimension in a more metric geometric set up. Another interesting result we establish that if two compact quantum metric spaces have quantum Gromov-Hausdorff distance zero, then their metric dimensions are equal (Theorem \ref{qGH vs Mdim}). This result is particularly helpful in preventing the collapse of quantum Gromov Hausdorff distance between compact quantum metric spaces as illustrated in Remark \ref{qGHcollapse}.

The methods developed here combine ideas from compact quantum metric spaces, growth theory of discrete groups, cocycle deformation theory, and finite-dimensional approximation techniques. Although our constructions apply to general twisted crossed products, they are particularly effective for groups whose growth is controlled by functions of the form $e^{Cn^\alpha}$, thereby encompassing polynomial, intermediate, and many exponential growth regimes within a unified framework.

\vspace{2mm}

{\bf Acknowledgement}: The first author acknowledges the financial support under the Senior Research Fellowship Scheme funded by UGC. The authors are grateful to Dr. Shubhabrata Das for several helpful discussions on geometric group theory. During the preparation of this manuscript, the authors used ChatGPT (OpenAI, GPT-5.6) as an interactive language model to assist with literature navigation. 



\section{Preliminaries}
\subsection{Operator Systems as compact quantum metric spaces}

We take most of these preliminaries on operator systems as compact quantum metric spaces from \cite{Kaad-Kyed}, although we shall make slight adjustments to some statements and definitions according to our need. Throughout this paper, an \emph{operator system} is a norm-closed self-adjoint subspace
\(X\) of a unital \(C^{*}\)-algebra \(A_{X}\) containing the unit \(1_{X}\).
Thus \(X\) is closed under the involution and contains the order unit inherited
from the ambient \(C^{*}\)-algebra, but need not be closed under multiplication. For generalities of operator systems, the reader is referred to \cite{arveson1990operator}.

A \emph{state} on \(X\) is a positive linear functional
\(\mu:X\to \mathbb{C}\) satisfying \(\mu(1_{X})=1\).
As in the \(C^{*}\)-algebra setting, every state has norm one, and hence the
state space \(S(X)\) is compact in the weak$^{\ast}$-topology.



Associated to an operator system \(X\) there is its selfadjoint part
\[
X_{\mathrm{sa}}=\{x\in X:x=x^{*}\},
\]
which forms a real order unit space with order and order unit inherited from
the ambient \(C^{*}\)-algebra \(A_{X}\). The positive cone is given by
\[
X_{\mathrm{sa}}^{+}
=
\{x\in X_{\mathrm{sa}}:x\geq 0 \text{ in } A_{X}\},
\]
and the distinguished order unit is \(1_{X}\).

The state spaces of \(X\) and \(X_{\mathrm{sa}}\) are naturally identified.
Indeed, every state on \(X\) restricts to a state on \(X_{\mathrm{sa}}\), and
every state on \(X_{\mathrm{sa}}\) extends uniquely to a state on \(X\) by
complex linearity. This correspondence allows one to freely pass between 
operator systems and order unit spaces, a viewpoint that is particularly useful
in the study of compact quantum metric spaces.
\begin{defn}
    Let $X$ be an operator system. A seminorm $L:X\rightarrow [0,+\infty]$ is said to be a Lipschitz seminorm if the following hold:\\
    \indent (i) $L$ is densely defined i.e. $\mathrm{Dom}(L):=\{x\in X: L(x)<+\infty\}$ is a norm-dense subspace of $X$.\\
    \indent (ii) Kernel of $L$ are the scalars i.e. $\{x\in X:L(x)=0\}=\mathbb{C}1_{X}$.\\
    \indent (iii) $L$ is invariant under the adjoint operation i.e. $L(x)=L(x^{\ast})$ for all $x\in X$.\\
\end{defn}
\begin{defn}
    Let $L:X\rightarrow[0,+\infty]$ be a Lipschitz seminorm on an operator system $X$. The Monge-Kantorovich distance $d_{L}:S(X)\times S(X)\rightarrow [0,+\infty]$ is defined by 
    \begin{displaymath}
        d_{L}(\mu,\nu):=\sup\limits_{L(x)\leq 1}\{\lvert\mu(x)-\nu(x)\rvert\}, \ \mu,\nu\in S(X).
    \end{displaymath}
\end{defn}
We remark that the Monge–Kantorovic metric  $d_{L}$ is not strictly speaking a metric since it can a priori take the value $+\infty$. However, this possibility is excluded when $(X,L)$ is a compact quantum metric space in the following sense:
\begin{defn}
\label{CQMSdefn}
    Let $L:X\rightarrow [0,+\infty]$ be a Lipschitz seminorm. We say that $(X,L)$ is a compact quantum metric space if the Monge-Kantorovich distance $d_{L}$ metrizes the weak*-topology of $S(X)$. In that case $L$ is said to be a Lip-norm.
\end{defn}
The following characterization of a Lip-norm on an operator system is very useful. For the proof in the case of pre $C^{\ast}$-algebras, the reader is referred to \cite{Ozawa-Rieffel}*{Proposition 1.3}. The same proof goes through verbatim if one replaces a pre $C^{\ast}$-algebra by an operator system. We state the theorem without proof.
\begin{thm}\label{totbouiff}
    Let $L$ be a Lipschitz seminorm on an operator system $X$ and let $\sigma\in S(X)$. Then $L$ is a Lip-norm if and only if the set $\{x\in X: L(x)\leq 1, \sigma(x)=0\}$ is norm totally bounded in X. 
\end{thm}
Given any Lipschitz seminorm $L$ on an operator system $X$ one can pass to the natural order unit space $X_{\mathrm{sa}}$ and induce a Lipschitz seminorm $L_{\mathrm{sa}}$ on $\mathrm{Dom}(L)_{\mathrm{sa}}=\mathrm{Dom}(L)\cap X_{\mathrm{sa}}$ in the sense of M. Rieffel. In fact, one has the following proposition:
\begin{prop}\label{operatorsystemtoorder}(\cite{Kaad-Kyed}*{Proposition 2.1.8})
    If $(X,L)$ is a compact quantum metric space then $(\mathrm{Dom}(L)_{\mathrm{sa}},L_{\mathrm{sa}})$ is an order unit compact quantum metric space in the sense of M.Rieffel.
\end{prop}
Now we recall the definition of metric dimension of a compact quantum metric space $(X,L)$ where $X$ is a complete operator system in our sense. We note that the original definition of metric dimension is due to David Kerr (\cite{Kerr}) in the set up of unital $C^{\ast}$ algebras. But it is easy to see that the definition has a straightforward generalization to operator spaces. We recall a few notations from \cite{Kerr}. For a normed linear space $(X,\| \cdot \|)$ (either an operator system or a Hilbert space for us), $\mathcal{F}(X)$ will denote the collection of all finite dimensional subspaces of $X$. If $Y,Z$ are subsets of $X$, then for $\delta>0$, the notation $Y\subseteq_{\delta} Z$ will mean that for every $y\in Y$, there is some $z\in Z$ such that $\|y-z\| <\delta$. For any subset $Y\subset X$, $D(Y,\delta)=\inf\{{\textrm{dim}}(Z):Z\in\mathcal{F}(X), Y\subseteq_{\delta}Z\}$, where $\textrm{dim}(Z)$ is the vector space dimension of $Z$. Then it is easy to see that if $Y_{1}\subseteq Y_{2}$, $D(Y_{1},\delta)\leq D(Y_{2},\delta)$ for any $\delta>0$. Now let $(X,L)$ be a compact quantum metric space on an operator system $X$ in the sense of Definition \ref{CQMSdefn}. Then we denote the set $\{x\in X:L(x)\leq 1\}$ by $\mathcal{L}_1$. Then we have the following:
\begin{defn}
    The metric dimension of a CQMS $(X,L)$ is defined to be
    \begin{displaymath}
    \mathrm{Mdim}_{L}(X):=\limsup\limits_{\delta\to 0^{+}} \frac{\log D(\mathcal{L}_{1},\delta)}{\log  \delta^{-1}}
    \end{displaymath}
\end{defn}
Now we recall the bi-Lipschitz equivalence which is the relevant equivalence in the context of CQMS theory.
\begin{defn}
    \label{bi-lipschitz_equivalence}(\cite{Kerr}*{Definition 2.8}) Let $X,Y$ be two operator systems with Lip-norms $L_{X}$, $L_{Y}$ respectively. A positive unital linear map $\phi:X\rightarrow Y$ is said to be Lipschitz if there is a $C\geq 0$, such that $L_{Y}(\phi(x))\leq C L_{X}(x)$ for all $x\in \mathrm{Dom}(L_{X})$. If $\phi$ is invertible and both $\phi,\phi^{-1}$ are Lipschitz then we say $\phi$ is bi-Lipschitz and in that case we say $(X,L_X)$ and $(Y,L_Y)$ are bi-Lipschitz equivalent.
\end{defn}
The metric dimension is an invariant for bi-Lipschitz equivalence. We recall the following theorem whose proof would be a slight modification of the existing proof because of our different choice of $\mathcal{L}_1$ and our choice of operator systems. It follows from the observation that positive unital maps are norm contractive for operator systems.
\begin{thm}
    \label{Metric_inv}(\cite{Kerr}*{Proposition 3.4}) Let $(X,L_X)$ and $(Y,L_{Y})$ be two bi-Lipschitz equivalent CQMS. Then 
    \begin{displaymath}
        \mathrm{Mdim}_{L_X}(X)=\mathrm{Mdim}_{L_{Y}}(Y).
    \end{displaymath}
\end{thm}
\subsection{Quantum Gromov-Hausdorff distance} We briefly recall the notion of quantum Gromov-Hausdorff distance between compact quantum metric spaces as discussed in the section 2.2 of \cite{Kaad-Kyed}. We would like to mention that this notion of quantum Gromov-Hausdorff distance is different from Latremoliere's quantum propinquity. Rather it is an adaptation of the quantum Gromov-Hausdorff distance defined originally by M. Rieffel (\cite{Rieffel-Gromov}) in the context pf order unit spaces. 
\begin{defn}
    Given two compact quantum metric spaces $(X,L)$ and $(Y,K)$, a Lipschitz seminorm $M:X\oplus Y\rightarrow[0,+\infty]$ is said to be admissible if $M$ is a Lip-norm on $X\oplus Y$, $\mathrm{Dom}(M)=\mathrm{Dom}(L)\oplus\mathrm{Dom}(K)$ and the quotient seminorms induced by $M_{\mathrm{sa}}$ via the coordinate projections 
    \begin{displaymath}
        \mathrm{Dom}(M)_{\mathrm{sa}}\rightarrow \mathrm{Dom}(L)_{\mathrm{sa}}, \quad \mathrm{Dom}(M)_{sa}\rightarrow \mathrm{Dom}(K)_{\mathrm{sa}}
    \end{displaymath}
    agree with $L_{\mathrm{sa}}$ and $K_{\mathrm{sa}}$ respectively.
\end{defn}
Whenever $L$ is an admissible Lipschitz seminorm it follows that the coordinate projections $X\oplus Y\rightarrow X$ and $X\oplus Y\rightarrow Y$ induce isometries from $S(X)\rightarrow S(X\oplus Y)$ and $S(Y)\rightarrow S(X\oplus Y)$ where the state spaces are equipped with the Monge-Kantorovich distance coming from the relevant Lip-norms. In particular, one can measure the Hausdorff distance between $S(X)$ and $S(Y)$ with respect to the Monge-Kantorovich distance $d_{L}$ in $S(X\oplus Y)$. Denoting this distance by $\mathrm{dist}^{d_{L}}_{H}(S(X),S(Y))$, one defines the quantum Gromov-Hausdorff distance between the compact quantum metric spaces $(X,L_X)$ and $(Y,L_Y)$ by
\begin{displaymath}
    \mathrm{dist}_{\mathrm {qGH}}((X,L),(Y,K)):=\mathrm{inf}\{\mathrm{dist}^{d_{L}}_{H}(S(X),S(Y)):M:X\oplus Y\rightarrow[0,\infty] \ \mathrm{admissible}\}.
\end{displaymath}
Using the adjoint invariance of Lip-norms, one can obtain the next theorem which essentially says that the quantum Gromov-Hausdorff distance between two compact quantum metric spaces can be recovered from their self-adjoint parts. For the proof, we refer the reader to \cite{Kaad-Kyed}*{Lemma 2.2.2}.

\begin{thm} \label{Rieffelequality}
   $\mathrm{dist}_{\mathrm {qGH}}((X,L),(Y,L))=\mathrm{dist}_{Q}((\mathrm{Dom}(L)_{\mathrm {sa}},L_{\mathrm{sa}}),(\mathrm{Dom}(K)_{\mathrm {sa}},K_{\mathrm {sa}}))$, where $\mathrm{dist}_{Q}$ denotes the distance between order unit compact quantum metric spaces in the sense of M. Rieffel. 
\end{thm}
Now we shall prove that if the quantum Gromov-Hausdorff distance between two compact quantum metric spaces are zero, then they have equal metric dimension. This will be particularly useful to prevent collapse of quantum Gromov Hausdorff distance. We start with the definition of the closure of a CQMS which is a straightforward generalization of the same concept introduced by M. Rieffel (\cite{Rieffel-Metric}) in the context of order unit spaces. Given a CQMS $(X,L)$, recall that the unit Lip-ball $\{x\in X, L(x)\leq 1\}$ is denoted by $\mathcal{L}_{1}$. We denote the closure of $\mathcal{L}_1$ in $X$ as usual by $\overline{\mathcal{L}}_1$. Then we define a seminorm $\overline{L}$ as the Minkowski functional of $\overline{\mathcal{L}}_1$ i.e. for $x\in X$,
\begin{displaymath}
    \overline{L}(x):=\inf\{r>0\ :\ x\in r\overline{\mathcal{L}}_1\}.
\end{displaymath}
Then it is straightforward to verify that $\overline{\mathcal{L}}_1=\{x\in X:\overline{L}(x)\leq 1\}$. Moreover, for any states $\mu,\nu\in S(X)$, we have
\begin{displaymath}
    d_{\overline{L}}(\mu,\nu)=\sup\{\lvert\mu(x)-\nu(x)\rvert:x\in\overline{\mathcal{L}}_1\}=\sup\{\lvert\mu(x)-\nu(x)\rvert:x\in\mathcal{L}_1\}=d_{L}(\mu,\nu),
\end{displaymath}
i.e. the Monge-Kantorovich distance on $S(X)$ induced by $L$ and $\overline{L}$ are same.

\begin{rem}
If $L$ is lower semicontinuous on $\mathrm {Dom}\ L,$ then $\overline {L} = L$ on $\mathrm {Dom}\ L.$ In other words, $\overline L$ is an extension of $L.$
\end{rem}

\begin{lem}
\label{Lip-Ext-Adj}
 The seminorm $\overline {L}$ is adjoint-invariant i.e. $\overline{L}(x)=\overline{L}(x^{\ast})$ for all $x\in X$. 
\end{lem}

\begin{proof}
Let $x \in X$ be such that $\overline {L} (x) = +\infty.$ If there exists some $r_0 \in \mathbb R^{+}$ such that $x^{\ast} \in r_0 \overline {\mathcal {L}_1},$ then there exists a sequence $x_n \in \mathcal L_1$ such that $r_0 x_n \rightarrow x^{\ast}.$ But then $r_0 x_n^{\ast} \rightarrow x.$ Since $L$ is adjoint-invariant, $x_n^{\ast} \in \mathcal {L}_1$ and consequently, $x \in r_0 \overline {\mathcal {L}_1}.$ But then $\overline {L} (x) \leq r_0 < \infty,$ a contradiction. This shows that there does not exist any $r \in \mathbb R^{+}$ such that $x^{\ast} \in r \overline {\mathcal {L}_1}$ and hence $\overline {L} \left (x^{\ast} \right ) = \infty.$

Now let us assume that $\overline {L} (x) < \infty.$ Let $\varepsilon > 0$ be taken arbitrarily. Then there exists $r_1 \in \mathbb R^{+}$ such that $b \in r_1 \overline {\mathcal {L}_1}$ and $r_1 < \overline {L} (x) + \varepsilon.$ So we can get hold of $x_n^{\prime} \in \mathcal L_1$ such that $r_1 x_n^{\prime} \rightarrow x.$ Therefore, $r_1 {x_n^{\prime}}^{\ast} \rightarrow x^{\ast}.$ Since $L$ is adjoint-invariant, ${x_n^{\prime}}^{\ast} \in \mathcal {L}_1$ and consequently, $x^{\ast} \in r_1 \overline {\mathcal {L}_1}.$ This shows that $\overline {L} \left (x^{\ast} \right ) \leq r_1 < \overline {L} (x) + \varepsilon.$ Since $\varepsilon > 0$ was arbitrary, letting $\varepsilon \rightarrow 0^{+},$ it follows that $\overline {L} \left (x^{\ast} \right ) \leq \overline {L} (x).$ Replacing $x$ by $x^{\ast},$ the desired adjoint invariance follows.
\end{proof}



\begin{prop}
    If $(X,L)$ is a compact quantum metric space then, $(X,\overline{L})$ is also a compact quantum metric space.
\end{prop}
\begin{proof}
    It is easy to see that $\mathrm {Dom}\ L \subseteq \mathrm {Dom}\ \overline {L}.$ Since $L$ is a Lip-norm on $X,$ $\mathrm {Dom}\ L$ is dense in $X$ and hence so is $\mathrm {Dom}\ \overline {L}.$ The adjoint-invariance of $\overline {L}$ follows from Lemma \ref{Lip-Ext-Adj}. Also since $d_{\overline {L}} = d_{L}$,  $\overline {L}$ has to have one dimensional kernel $\mathbb C 1_{X},$ for otherwise $d_{\overline {L}}$ takes the value $+\infty,$ which is a contradiction to the fact that $d_{L}$ always takes finite values. Finally,  $d_{\overline {L}}$ metrizes the weak$^{\ast}$-topology of $S (X),$ since so does $d_{L}.$ This completes the proof. 
\end{proof}
\begin{lem} \label{Mdim-Inv}
Let $(X,L)$ be a compact quantum metric space. Then $\mathrm {Mdim}_{L} (X) = \mathrm {Mdim}_{\overline {L}} (X).$
\end{lem}

\begin{proof}
Let us choose $\delta > 0$ arbitrarily.  Then $\overline{\mathcal{L}}_1$ is the unit Lip-ball with respect to the Lip-norm $\overline{L}$. Therefore, $D \left (\mathcal {L}_1, \delta \right ) \leq D \left (\overline{\mathcal {L}}_1, \delta \right ).$ Consequently, $\mathrm {Mdim}_{L} (X) \leq \mathrm {Mdim}_{\overline {L}} (X).$

To show the reverse inequality, first note that since $L$ is adjoint-invariant, by \cite{Kerr}*{Proposition 3.2}, $D \left (\mathcal {L}_1, \delta \right ) < \infty.$ Then there exists a finite dimensional subspace $Y$ of $X$ such that $\mathcal {L}_1 \subseteq_{\delta} Y$ and $D \left (\mathcal {L}_1, \delta \right ) = \dim Y.$ We claim that $\overline{\mathcal{L}}_1 \subseteq_{2 \delta} Y.$ To see that, let us take any $x \in \overline{\mathcal{L}}_1.$ Then there exist $z \in \mathcal {L}_1$ and $y \in Y$ such that $\left \|x - z \right \| < \delta$ and $\left \|z - y \right \| < \delta.$ Then by the triangle inequality we have
$$\left \|x - y \right \| \leq \left \|x - z \right \| + \left \|z - y \right \| < \delta + \delta = 2 \delta.$$
Since this holds for any $x \in \overline{\mathcal{L}}_1,$ the desired claim follows. This shows that $$D \left (\overline{\mathcal{L}}_1, 2 \delta \right ) \leq \dim Y = D \left (\mathcal {L}_1, \delta \right ).$$ Now dividing both sides by $\log \delta^{-1}$ and letting $\delta \rightarrow 0^{+},$ the required reverse inequality follows. 
\end{proof}
Now we shall prove the main theorem of this subsection. We are going to need the following technical lemma.
\begin{lem}
Let $L$ be a Lipschitz seminorm on $X$ and $\overline {L}$ be its extension to $X.$ Let $(\overline {L})_{\mathrm {sa}}$ be the restriction of $\overline {L}$ to $X_{\mathrm {sa}}.$ Let $L_{\mathrm {sa}}$ be the restriction of $L$ to $\left (\mathrm {Dom}\ L \right )_{\mathrm {sa}}$ and $\overline {L_{\mathrm {sa}}}$ be its extension to $X_{\mathrm {sa}}.$ Then $(\overline {L})_{\mathrm {sa}} = \overline {L_{\mathrm {sa}}}.$
\end{lem}

\begin{proof}
Let $\mathcal L_1$ and $\left (\mathcal {L}_1 \right )_{\mathrm {sa}}$ be the unit Lip-ball with respect to $L$ and $\overline {L}_{\mathrm {sa}}$ respectively. Since $L$ is the restriction of $\overline {L}$ to $\mathrm {Dom}\ L,$ it follows that $\left (\mathcal {L}_1 \right )_{\mathrm {sa}} \subseteq \mathcal {L}_1.$ Consequently, $(\overline {L})_{\mathrm {sa}} \leq \overline {L_{\mathrm {sa}}}$ on $X_{\mathrm {sa}}.$

To show the reverse inequality, we take $b \in X_{\mathrm {sa}}$ and $\varepsilon > 0$ arbitrarily. If $(\overline {L})_{\mathrm {sa}} (b) < \infty,$ then there exists $r_0 \in \mathbb R^{+}$ such that $r_0 < (\overline {L})_{\mathrm {sa}} (b) + \varepsilon$ and $b \in r_0 \overline {\mathcal {L}_1}.$ Then there exists a sequence $a_n \in \mathcal {L}_1$ such that $r_0 a_n \to b.$ Let $b_n = \frac {a_n + a_n^{\ast}} {2}.$ Since $b$ is self-adjoint, $r_0 b_n \rightarrow b$ and moreover, since $L$ is adjoint-invariant, $L \left (b_n \right ) \leq L \left (a_n \right ) \leq 1.$ Therefore $b_n \in \left (\mathcal {L}_1 \right )_{\mathrm {sa}}$ for all $n \geq 1$ and $b \in r_0 \overline {\left (\mathcal {L}_1 \right )_{\mathrm {sa}}}.$ Therefore, $\overline {L_{\mathrm {sa}}} (b) \leq r_0 < (\overline {L})_{\mathrm {sa}} (b) + \varepsilon.$ Since this holds for any $\varepsilon > 0,$ letting $\varepsilon \rightarrow 0^{+},$ it follows that $\overline {L_{\mathrm {sa}}} (b) \leq (\overline {L})_{\mathrm {sa}} (b).$ If $(\overline {L})_{\mathrm {sa}} (b) = \infty,$ then the reverse inequality is automatically satisfied. This completes the proof.
\end{proof}

\begin{thm} \label{qGH vs Mdim}
Let $(X,L)$ and $(Y,K)$ be compact quantum metric spaces. Then \begin{displaymath}\mathrm {dist}_{\mathrm {qGH}} ((X, L), (Y, K)) = 0\Rightarrow \mathrm {Mdim}_{L} (X) = \mathrm {Mdim}_{K} (Y).\end{displaymath}
\end{thm}

\begin{proof}
Let $\mathrm {dist}_{\mathrm {qGH}} ((X, L), (Y, K)) = 0.$ Then by Theorem \ref{Rieffelequality}, $$\mathrm {dist}_Q \left (\left (\left (\mathrm {Dom}\ L \right )_{\mathrm {sa}}, L_{\mathrm {sa}} \right ), \left (\left (\mathrm {Dom}\ K \right )_{\mathrm {sa}}, K_{\mathrm {sa}} \right ) \right ) = 0.$$
Since $\left (\mathrm {Dom}\ L \right )_{\mathrm {sa}}$ and $\left (\mathrm {Dom}\ L \right )_{\mathrm {sa}}$ are both order-unit spaces, by \cite{Rieffel-Gromov}*{Theorem 7.8}, there exists an order preserving bijection $\varphi : X_{\mathrm {sa}} \rightarrow Y_{\mathrm {sa}}$ such that $\overline {K_{\mathrm {sa}}} \circ \varphi = \overline {L_{\mathrm {sa}}}.$
Then by complexifying $\varphi,$ it will have order-preserving bijective extension $\Phi : X \rightarrow Y$ given by $\Phi (x) = \varphi \left (x_1 \right ) + i \varphi \left (x_2 \right ),$ where $x_1 = \frac {x + x^{\ast}} {2}$ and $x_2 = \frac {x - x^{\ast}} {2 i}.$ We claim that $\Phi$ establishes bi-Lipschitz equivalence between the CQMSs $(X, \overline {L})$ and $(Y, \overline {K})$. 

To show that, let us take $x \in \mathrm {Dom}\ \overline {L}.$ Let $x_1$ and $x_2$ be its real and imaginary parts as above. Then by the adjoint invariance of $\overline{L}$, $x_1,x_2\in \mathrm{Dom}\big((\overline{L})_{\mathrm{sa}}\big)$ and $(\overline{L})_{\mathrm{sa}}(x_1)\leq \overline{L}(x),(\overline{L})_{\mathrm{sa}}(x_2)\leq\overline{L}(x)$. We have
\Bea
\overline {K} \left (\Phi (x) \right ) & = & \overline K \left (\varphi \left (x_1 \right ) + i\ \varphi \left (x_2 \right ) \right ) \\ & \leq & \overline {K} \left (\varphi \left (x_1 \right ) \right ) + \overline {K} \left (\varphi \left (x_2 \right ) \right ) \\ & = & (\overline {K})_{\mathrm {sa}} \left (\varphi \left (x_1 \right ) \right ) + (\overline {K})_{\mathrm {sa}} \left (\varphi \left (x_2 \right ) \right ) \\ & = & \overline {K_{\mathrm {sa}}} \left (\varphi \left (x_1 \right ) \right ) + \overline {K_{\mathrm {sa}}} \left (\varphi \left (x_2 \right ) \right ) \\ & = & \overline {L_{\mathrm {sa}}} \left (x_1 \right ) + \overline {L_{\mathrm {sa}}} \left (x_2 \right ) \\ & = & (\overline {L})_{\mathrm {sa}} \left (x_1 \right ) + (\overline {L})_{\mathrm {sa}} \left (x_2 \right ) \\ & = & \overline {L} \left (x_1 \right ) + \overline {L} \left (x_2 \right ) \\ & \leq & 2 \overline {L} (x). 
\Eea
For $x \in X \setminus \mathrm {Dom}\ \overline {L},$ the above inequality is automatically satisfied. Thus for all $x \in X$ we have $\overline {K} \left (\Phi (x) \right )  \leq 2\overline {L} (x).$ Similarly, we can show that $\overline {L} \left (\Phi^{-1} (y) \right ) \leq 2 \overline {K} (y)$ for all $y \in Y.$ This shows that $(X, \overline{L})$ and $(Y, \overline{K})$ are bi-Lipschitz equivalent, as claimed.

Since metric dimension remains invariant under the bi-Lipschitz equivalence, it follows that $\mathrm {Mdim}_{\overline {L}} (X) = \mathrm {Mdim}_{\overline {K}} (Y).$ Now the desired result follows by virtue of Lemma \ref{Mdim-Inv}.
\end{proof}
\subsection{Cocycle twisted crossed product \texorpdfstring{$\mathrm{C}^{\ast}$}{C*-} algebras} The underlying operator systems of our main examples of compact quantum metric spaces will be cocycle twisted reduced crossed product $C^{\ast}$-algebras. Therefore, we recall the basics of cocycle twisted crossed products in this subsection.

Let $(A,\rho,\Gamma)$ be a $C^{\ast}$-dynamical system where $A$ is a unital $C^{\ast}$-algebra; $\Gamma$ is a discrete group and $\rho:\Gamma\rightarrow\mathrm{Aut}(A)$ is an action of $\Gamma$ on $A$; $\sigma$ be a unitary $2$-cocycle on $\Gamma$. Then we shall define the reduced $C^{\ast}$-algebra of crossed product of $A$ by $\Gamma$ twisted by the cocycle $\sigma$ as the reduced cross sectional algebra of a Fell bundle. For Fell bundles and the associated reduced cross sectional algebras, the reader is referred to \cite{Exel-Fell}*{Chapter 17}.

The relevant Fell bundle is $\mathcal{G}=\bigsqcup\limits_{t \in \Gamma} G_t$ over $\Gamma$ where each fibre at $t\in\Gamma$ is defined to be $G_{t}=\{\delta_t a:a\in A\}$. The Banach space norm on $G_t$ is then simply given by $\left \|\delta_t a \right \| = \|a\|_A.$
We define the multiplication and involution on $\mathcal{G}$ in terms of the action $\rho$ and the unitary $2$-cocycle $\sigma : \Gamma \times\Gamma\rightarrow \mathbb{T}$ as follows $:$

For $\delta_s a \in G_s$ and $\delta_t b \in G_t,$ we define
\Bea
    (\delta_s a) \cdot (\delta_t b) & = & \delta_{st} \sigma(s, t) \rho_{t^{-1}}(a) b \\
    (\delta_s a)^* & = & \delta_{s^{-1}} \overline{\sigma(s^{-1}, s)}\rho_s(a^*)
\Eea
Then the reduced $C^{\ast}$-algebra of cocycle twisted crossed product $A\rtimes_{r,\rho,\sigma}\Gamma$ is defined to be the cross-sectional algebra of the above Fell bundle. The $C^{\ast}$-algebra has the following dense $\ast$-subalgebra which is going to be important for us:
\begin{displaymath}
    C_{c}(\Gamma,A):=\left \{\sum_{s\in\Gamma}\delta_{s} a_{s}:a_s\in A \right \},
\end{displaymath}
where $s$ runs over a finite subset of $\Gamma$.\\

We are going to need a field of $C^{\ast}$-algebras. To that end, let $\Omega$ be a compact Hausdorff space, $A$ be a unital $C^{\ast}$-algebra and $\Gamma$ be a discrete group acting on $A$ via the  automorphism $\rho.$ Let $\left \{\sigma_{\theta} \right \}_{\theta \in \Omega}$ be a strongly continuous family of $2$-cocycles on $\Omega.$ By a strongly continuous family we mean that for each fixed $s,t\in\Gamma$, the map $\theta\mapsto\sigma_{\theta}(s,t)$ is continuous on $\Omega$. Consider the $C^{\ast}$-algebra $B = C (\Omega, A),$ which is unital since so is $A.$ For a fixed $\theta \in \Omega,$ define $J_{\theta}^{0} : = \{f \in C (\Omega)\ :\ f (\theta) = 0 \}$ and $J_{\theta} : = \{g \in B\ :\ g (\theta) = 0 \}.$ Note that $J_{\theta}$ is a Banach module over $J_{\theta}^{0}$ by pointwise multiplication.

\begin{lem} \label{C-H Equality}
$J_{\theta} = J_{\theta}^{0} \cdot B$ for any $\theta \in \Omega.$
\end{lem}

\begin{proof}
Let $X = \Omega \setminus \{\theta\}$ and $\Lambda$ be the directed set consisting of compact subsets of $X,$ ordered by set inclusions. For $K \in \Lambda,$ let $U_K$ be an open neighbourhood of $K$ such that $\overline {U_K}$ is compact. By Urysohn's lemma, there exists $e_K \in C_c (X) \subseteq C_0 (X)$ such that $0 \leq e_K \leq 1,$ $e_K \rvert_{K} \equiv 1$ and $e_K \rvert_{X \setminus U_K} = 0$ so that $\mathrm {supp} \left (e_K \right ) \subseteq \overline {U_K}.$ We claim that $\left \{e_K \right \}_{K \in \Lambda}$ is a bounded approximate identity of $C_0 (X).$ Let us choose $\varepsilon > 0$ and let $f \in C_0 (X).$ Then there exists $K_0 \in \Lambda$ such that $\left \lvert f (x) \right \rvert < \varepsilon$ for all $x \notin K_0.$ Now fix any $K \in \Lambda$ arbitrarily with $K \supseteq K_0.$ Then for all $y \in K,$ we have $e_K (y) f (y) = f (y)$ and for all $y \notin K,$ we have $\left (1 - e_K (y) \right ) \left \lvert f (y) \right \rvert \leq \left \lvert f (y) \right \rvert < \varepsilon.$ This shows that for any $K \in \Lambda$ with $K \supseteq K_0,$ we have $\left \|e_K f - f \right \|_{\infty} \leq \varepsilon.$ Thus $\left \{e_K \right \}_{K \in \Lambda}$ is an approximate identity for $C_0 (X),$ as claimed. Since $C_0 (X)$ is canonically identified with $J_{\theta}^{0},$ it follows that $\left \{\widetilde {e_K} \right \}_{K \in \Lambda}$ is an approximate identity for $J_{\theta}^{0},$ where
$$\widetilde {e_K} (x) = \begin{cases} e_K (x), \quad \mathrm{if}\ x \in X, \\ 0, \quad \hspace{8mm} \mathrm{if}\ x = \theta. \end{cases}$$
Now for any $g \in J_{\theta},$ we can get hold of an open neighbourhood $V_{\theta} \subseteq \Omega$ of $\theta$ such that for all $x \in V_{\theta},$ we have $\left \|g (x) \right \| < \varepsilon.$ Let $K_{\theta} : = \Omega \setminus V_{\theta},$ which, being a closed subset of a compact set, is compact and hence $K_{\theta} \in \Lambda.$ Fix some $K \in \Lambda$ with $K \supseteq K_{\theta}$ arbitrarily. Then for all $y \in K,$ we have $\widetilde {e_K} (y) g (y) = g (y)$ and for all $y \notin K,$ we have $\left (1 - \widetilde {e_K} (y) \right ) \left \|g (y) \right \| \leq \left \|g (y) \right \| < \varepsilon.$ This shows that $\lim\limits_{K \in \Lambda} \widetilde {e_K} g = g.$ Thus by Cohen-Hewitt Factorization Theorem we have 
$$J_{\theta} = J_{\theta}^{0} \cdot J_{\theta} \subseteq J_{\theta}^{0} \cdot B \subseteq J_{\theta}.$$ This shows that $J_{\theta}^{0} \cdot B = J_{\theta},$ as required.
\end{proof}

Let $B = C (\Omega, A).$ Then $\rho$ induces an action $\hat {\rho}$ of $\Gamma$ on $B$ given by $$\hat {\rho}_t (g) (\theta) = \rho_t (g (\theta)),$$ $t \in \Gamma,$ $g \in B$ and $\theta^{\prime} \in \Omega.$ Consider the integrated $2$-cocycle $\Sigma : \Gamma \times \Gamma \rightarrow C (\Omega, \mathbb T)$ defined by $$\Sigma (s, t) (\theta) = \sigma_{\theta} (s, t),$$ for $s, t \in \Gamma$ and $\theta \in \Omega.$ Define a map $\Phi : C (\Omega) \rightarrow B$ by $\Phi (f) (\theta) = f (\theta) 1_A,$ for all $f \in C (\Omega)$ and for all $\theta^{\prime} \in \Omega.$ Since scalar multiples of identity are in the center of $B,$ it follows that $\Phi$ maps $C (\Omega)$ into the center of $B.$ Moreover, since $C (\Omega)$ and $\Phi$ are both unital, $\Phi$ is also non-degenerate. This shows that $B$ is a $C (\Omega)$-algebra. Let us consider the twisted crossed product $\mathcal B : = B \rtimes_{r, \hat {\rho}, \Sigma} \Gamma.$ Note that $\Phi$ induces a non-degenerate $\ast$-homomorphism $\tilde {\Phi} : C (\Omega) \rightarrow \mathcal B,$ given by $\tilde {\Phi} (f) = \delta_e \Phi (f),$ for all $f \in C (\Omega).$ Also for any $t \in \Gamma,$ $f \in C (\Omega)$ and $\theta^{\prime} \in \Omega,$ we have 
$$\hat {\rho}_t (\Phi (f)) (\theta) = \rho_t (\Phi (f) (\theta)) = \rho_t (f (\theta) 1_A) = f (\theta) 1_A = \Phi (f) (\theta).$$
This shows that the image of $\Phi$ is invariant under the action $\hat {\rho}.$ In other words, the image of $\tilde {\Phi}$ is in the center of $\mathcal B.$ So $\mathcal B$ also has the structure of a $C (\Omega)$-algebra. 

Recall that for any $\theta \in \Omega$ the fibre of the $C(\Omega)$-algebra $\mathcal{B}$ at $\theta\in\Omega$ is the quotient $\mathcal{B}/I_{\theta},$ where $I_{\theta}$ is the ideal in $\mathcal{B}$ defined by 
$$I_{\theta} : = \overline {\mathrm {span} \left \{\tilde {\Phi} (f) \xi\ :\ f \in J_{\theta}^{0},\ \xi \in \mathcal B \right \}}.$$
The next lemma helps us to identify the ideal $I_{\theta}$.
\begin{lem}\label{Itheta}
$I_{\theta} = J_{\theta} \rtimes_{r, \hat {\rho},\Sigma} \Gamma$ for all $\theta \in \Omega.$
\end{lem}

\begin{proof}
Let $f \in J_{\theta}^{0}$ and $\xi \in C_c (\Gamma, B)$. Since $\Phi (f)$ is invariant under the action $\hat {\alpha}$ of $\Gamma$ on $B$, it follows that for any $t \in \Gamma$,
$$(\tilde {\Phi} (f) * \xi) (t) = \Phi (f) \xi (t) \in B.$$
Also,
$$\Phi (f) \xi (t) (\theta) = f (\theta) \xi (t) (\theta) = 0.$$
So, $(\tilde {\Phi} (f) * \xi) (t) \in J_{\theta}$ for any $t \in \Gamma$. Consequently, $\tilde {\Phi} (f) * \xi \in C_c (\Gamma, J_{\theta})$. Now for any fixed $f \in J_{\theta}^{0}$ and $\eta \in \mathcal B$, $\tilde{\Phi} (f) * \eta$ can be approximated by elements of the form $\tilde {\Phi} (f) * \xi$ in the reduced norm, for $\xi \in C_c (\Gamma, B)$, and hence $\tilde {\Phi} (f) * \eta \in J_{\theta} \rtimes_{r, \hat {\rho}, \Sigma} \Gamma$ for any $f \in J_{\theta}^{0}$ and for any $\eta \in \mathcal B$. This, in turn, implies that $I_{\theta} \subseteq J_{\theta} \rtimes_{r, \hat {\rho}, \Sigma} \Gamma$.

\vspace{2mm}

For the reverse inclusion, we will make use of Lemma \ref{C-H Equality}. Fix some $t_0 \in \Gamma$ and $g \in J_{\theta}$. Then by virtue of Lemma \ref{C-H Equality}, there exists some $h \in J_{\theta}^{0}$ and $b \in B$ such that $g = h \cdot b = \Phi (h) b$. Since $\Phi (C (\Omega))$ is invariant under the action $\hat {\rho}$ of $\Gamma$ on $B$, it follows that
$$\delta_{t_0} g = (\delta_e h) * (\delta_{t_0} b) = \tilde {\Phi} (h) * (\delta_{t_0} b).$$
This shows that $\delta_{t_0} g \in \{\tilde{\Phi} (f) * \eta\ :\ f \in J_{\theta}^{0},\ \eta \in \mathcal B\}$. Therefore, 
$$C_c (\Gamma, J_{\theta}) \subseteq \overline{\mathrm {span} \{\tilde {\Phi} (f) * \eta\ : f \in J_{\theta}^{0},\ \eta \in \mathcal {B}\}} = I_{\theta}.$$ 
Consequently, $J_{\theta} \rtimes_{r, \hat {\alpha}, \Sigma} \Gamma \subseteq I_{\theta},$ as required.
\end{proof}

To pass an exact sequence of $C^{\ast}$-algebras into the exact sequence of the corresponding reduced crossed products, we use the machinery of Fell bundles. Passing to reduced cross-sectional algebras in general does not preserve exact sequences, but Exel’s theorem provides that reduced crossed product of an exact sequence is exact whenever the acting discrete group is exact \cite{Exel-Exact}*{Theorem 4.4.}.

We construct a Fell bundle $$\mathcal{E} = \bigsqcup\limits_{t \in \Gamma} E_t$$ over $\Gamma,$ where the fiber $E_t$ over $t \in \Gamma$ consists of elements of the form $\delta_t a$ with $a \in B.$ The Banach space norm on $E_t$ is then simply given by $\left \|\delta_t a \right \| = \|a\|_B.$
We define the multiplication and involution on $\mathcal{E}$ in terms of the action $\hat{\rho}$ and the integrated $2$-cocycle $\Sigma : \Gamma \times \Gamma \to C(\Omega, \mathbb{T})$ as follows $:$

For $\delta_s a \in E_s$ and $\delta_t b \in E_t,$ we define
\Bea
    (\delta_s a) \cdot (\delta_t b) & = & \delta_{st} \Sigma(s, t) \hat{\rho}_{t^{-1}}(a) b \\
    (\delta_s a)^* & = & \delta_{s^{-1}} \overline{\Sigma(s^{-1}, s)} \hat{\rho}_s(a^*)
\Eea
Note that the reduced cross-sectional algebra $C_r^{\ast}(\mathcal{E})$ is, by definition, the twisted reduced crossed product $B \rtimes_{r, \hat{\rho}, \Sigma} \Gamma.$ Similarly, by restricting the fibers to the ideal $J_{\theta},$ we obatin an ideal $\mathcal{J}_{\theta}$ of the Fell bundle $\mathcal E.$ Finally, we define the quotient Fell bundle $\mathcal{Q}_{\theta} = \mathcal E/ \mathcal J_{\theta},$ where each fiber is the Banach space $E_t/ \left (J_{\theta} \right )_t \cong \delta_t A.$ Then we have the following lemma$:$

\begin{lem}
Let $\mathcal F_{\theta}$ be the semidirect product bundle associated with the $C^{\ast}$-dynamical system $\left (A, \rho, \sigma_{\theta} \right )$ whose fibers are $F_t = \delta_t A.$ Then the quotient semidirect product bundle $\mathcal E/ \mathcal J_{\theta} \cong \mathcal F_{\theta}.$ Consequently, $C_r^{\ast} \left (\mathcal E/ \mathcal J_{\theta} \right ) \cong A \rtimes_{r, \rho, \sigma_{\theta}} \Gamma.$
\end{lem}

\begin{proof}
We define a fiber-wise map $\Phi : \mathcal{E} \to \mathcal{F}_{\theta}$. For each $t \in \Gamma,$ the map on the fiber $\Phi_t : E_t \to F_t$ is given by
$$\Phi_t(\delta_t b) = \delta_t \mathrm{ev}_{\theta}(b).$$
Because $\mathrm{ev}_{\theta} : B \to A$ is a surjective linear map, $\Phi$ is clearly a surjective map of semidirect product bundles. 

We now verify that $\Phi$ is a morphism of semidirect product bundles by checking that it preserves the multiplication and involution. Let $\delta_s a, \delta_t b \in \mathcal{E}.$ Their product in $\mathcal E$ is given by
$$(\delta_s a) \cdot_{\mathcal{E}} (\delta_t b) = \delta_{st} \big( \Sigma(s, t) \hat{\rho}_{t^{-1}}(a) b \big).$$
Since $\mathrm{ev}_{\theta}$ is a $*$-homomorphism, it follows that
\Bea
\Phi_{st} \big( (\delta_s a) \cdot_{\mathcal{E}} (\delta_t b) \big) & = & \delta_{st} \mathrm{ev}_{\theta} \big( \Sigma(s, t) \hat{\rho}_{t^{-1}}(a) b \big) \\
& = & \delta_{st} \big( \mathrm{ev}_{\theta}(\Sigma(s, t)) \cdot \mathrm{ev}_{\theta}(\hat{\rho}_{t^{-1}}(a)) \cdot \mathrm{ev}_{\theta}(b) \big) \\ & = & \delta_{st} \big( \sigma_{\theta}(s, t) \rho_{t^{-1}}(\mathrm{ev}_{\theta}(a)) \mathrm{ev}_{\theta}(b) \big).
\Eea
On the other hand, we have
\Bea
\Phi_s(\delta_s a) \cdot_{\mathcal{F}_{\theta}} \Phi_t(\delta_t b) & = & (\delta_s \mathrm{ev}_{\theta}(a)) \cdot_{\mathcal{F}_{\theta}} (\delta_t \mathrm{ev}_{\theta}(b)) \\
& = & \delta_{st} \big( \sigma_{\theta}(s, t) \rho_{t^{-1}}(\mathrm{ev}_{\theta}(a)) \mathrm{ev}_{\theta}(b) \big).
\Eea
This shows that $\Phi$ preserves the multiplication of the semidirect product bundles. An analogous calculation shows that $\Phi$ also preserves the involution, i.e., $\Phi$ is a surjective morphism of semidirect product bundles.
The kernel of this morphism is gievn by
$$\ker(\Phi) = \{ \delta_t b \in \mathcal{E} \mid \mathrm{ev}_{\theta}(b) = 0_A, t \in \Gamma \} = \{ \delta_t b \in \mathcal{E} \mid b \in J_{\theta}, t \in \Gamma \}.$$
This is exactly the ideal Fell bundle $\mathcal{J}_{\theta}.$ So by the First Isomorphism Theorem for Fell bundles, we obtain a canonical isomorphism $\mathcal{E} / \mathcal{J}_{\theta} \cong \mathcal{F}_{\theta}.$ 

\end{proof}

With these identifications at our disposal, we have the following exact sequence of twisted crossed product $C^{\ast}$-algebras.

\begin{lem} \label{Fell_Exactness}
Let $\Gamma$ be an {\bf exact} discrete group. Then for any $\theta \in \Omega,$ we have the following short exact sequence of twisted crossed product $C^{\ast}$-algebras, 
$$ 0 \rightarrow J_{\theta} \rtimes_{r, \hat{\rho}, \Sigma} \Gamma \hookrightarrow B \rtimes_{r, \hat{\rho}, \Sigma} \Gamma \xrightarrow{\pi_{\theta}} A \rtimes_{r, \rho, \sigma_{\theta}} \Gamma \rightarrow 0,$$
where $\pi_{\theta}$ is a $\ast$-homomorphism induced from $\mathrm {ev}_{\theta} : B \rightarrow A.$
\end{lem}

\begin{proof}
By the previous lemma, we have got hold of a short exact sequence of semidirect product bundles, namely
$$0 \rightarrow \mathcal J_{\theta} \hookrightarrow \mathcal E \xrightarrow {\Phi} \mathcal F_{\theta} \rightarrow 0.$$
Since $\Gamma$ is an exact discrete group, a fundamental result due to Exel \cite{Exel-Exact}*{Theorem 4.4.} ensures that the functor $C_r^{\ast} (\cdot)$ mapping semidirect product bundles into its associated reduced cross sectional algebra is exact, i.e., we have the following short exact sequence of $C^{\ast}$-algebras
$$0 \rightarrow C_r^{\ast}(\mathcal{J}_{\theta}) \hookrightarrow C_r^{\ast}(\mathcal{E}) \xrightarrow{\pi_{\theta}} C_r^{\ast}(\mathcal{F}_{\theta}) \rightarrow 0,$$
where $\pi_{\theta}$ is the integrated map induced by $\Phi.$ The required result then follows by identifying the reduced cross sectional algebras with the twisted reduced crossed product $C^{\ast}$-algebras.
\end{proof}

\begin{thm} \label{upper semicontinuity}
Let $(A,\rho,\Gamma)$ be a $C^{\ast}$-dynamical system where $A$ is a unital $C^{\ast}$-algebra and $\Gamma$ be an {\bf exact} discrete group; $\{\sigma_{\theta}\}_{\theta\in\Omega}$ is a strongly continuous one parameter family of unitary $2$-cocycles on $\Gamma$ for some compact parameter space $\Omega$. Then the family $\{A \rtimes_{r, \rho, \sigma_\theta} \Gamma\}_{\theta\in\Omega}$ is an upper semicontinuous family of $C^{\ast}$-algebras on $\Omega$ i.e. for any $v\in B \rtimes_{r, \rho, \Sigma} \Gamma$, the map $\theta\mapsto \left \| \pi_{\theta} (v) \right \|_{\theta}$ is upper semicontinuous on $\Omega,$ where $\| \cdot \|_{\theta}$ denotes the $C^{\ast}$-norm of $A \rtimes_{r, \rho, \sigma_\theta} \Gamma.$ Consequently, for any $\mathfrak{a} \in C_c (\Gamma, A),$ the map $\theta \mapsto \left \|\mathfrak{a} \right \|_{\theta}$ is upper semicontinuous on $\Omega.$
\end{thm}

\begin{proof}
We have already seen that $B \rtimes_{r, \hat{\rho}, \Sigma} \Gamma$ is a $C(\Omega)$-algebra. For any $\theta\in\Omega$, the fibre of this $C(\Omega)$-algebra at $\theta$ is $B \rtimes_{r, \hat{\rho}, \Sigma} \Gamma /I_{\theta}$. But $I_{\theta}$ is isomorphic to $J_{\theta}\rtimes_{r,\hat{\rho},\Sigma}\Gamma$ by Lemma \ref{Itheta}. Then the fibre is isomorphic to $A\rtimes_{r,\rho,\sigma_{\theta}}\Gamma$ by the previous Lemma \ref{Fell_Exactness}. Therefore $\{A\rtimes_{r,\rho,\sigma_{\theta}}\Gamma\}_{\theta\in\Omega}$ is an upper semicontinuous family of $C^{\ast}$-algebras by \cite{Dadarlat-Winter}*{Lemma 3.2 (i)}.  
\end{proof}
Note that the lower semicontinuity of the family $\{A\rtimes_{r,\rho,\sigma_{\theta}}\Gamma\}_{\theta\in\Omega}$ is automatic (even without the assumption of exactness of the group $\Gamma$) by \cite{Rieffel-Cont}*{Theorem 2.5}. Therefore, we have the following corollary which will be crucial for our continuity results of Subsection 3.1:
\begin{cor}\label{continuousfielCalgebras}
    Under the assumptions of the previous theorem, the family $\{A\rtimes_{r,\rho,\sigma_{\theta}}\Gamma\}_{\theta\in\Omega}$ is a continuous family of $C^{\ast}$-algebras. In particular, for any $\mathfrak{a}\in C_{c}(\Gamma,A)$, the function $\theta\mapsto\lvert\lvert\mathfrak{a}\rvert\rvert_{\theta}$ is continuous.
\end{cor}
\section{Main results}

\subsection{The family of compact quantum metric spaces} \label{Lip-norm defn}
Let $(A,\rho,\Gamma)$ be a $C^{\ast}$-dynamical system where $A$ is a unital $C^{\ast}$-algebra, $\Gamma$ is a discrete group equipped with a normalized $2$-cocycle $\sigma : \Gamma \times \Gamma \rightarrow \mathbb T.$  Let $\lambda$ denote the growth function of $\Gamma$ with respect to some length function $\ell.$ The discrete group $\Gamma,$ which is considered in this context, is assumed to have either subexponential growth or exponential growth, i.e., there always exist some $0 < \alpha \leq 1$ and a constant $C > 0$ such that the growth function $\lambda (n) \leq e^{C n^{\alpha}}$ for all $n \geq 1.$ For $\beta, \gamma \in \mathbb R,$ we consider a two parameter family of seminorms $L_{\ell}^{\beta, \gamma}$ on the associated $\sigma$-twisted reduced crossed product $A_{\sigma} : = A \rtimes_{r, \rho, \sigma} \Gamma$ as follows $:$
$$L^{\beta, \gamma}_{\ell} (\mathfrak{a}) = \begin{cases} \left (\sum\limits_{t \in \Gamma \setminus \{e \}} e^{2 \gamma \ell (t)^{\beta}} \left \|a_t \right \|^{2} \right )^{\frac {1} {2}}, \quad \text{if}\ \mathfrak{a} = \sum\limits_{t \in \Gamma} \delta_t a_t \in C_c (\Gamma,  A, \sigma), \\ \infty, \quad \mathrm{otherwise}. \end{cases}$$
Given such a two parameter family of seminorms and a Lipschitz seminorm $L_A$ on the base algebra $A,$ we define a two parameter family of stratified seminorms $L_{\ell}^{S, \beta, \gamma}$ on $A \rtimes_{r, \rho, \sigma} \Gamma$ in the following way $:$
$$L^{S, \beta, \gamma}_{\ell} (\mathfrak{a}) = \begin{cases} \max \left \{L^{\beta, \gamma}_{\ell} (\mathfrak{a}), \sup\limits_{t \in \Gamma} L_A \left (a_t \right ) \right \}, \quad \text {if}\ \mathfrak{a}= \sum\limits_{t \in \Gamma} \delta_t a_t \in C_c (\Gamma, \mathcal A, \sigma), \\ \infty, \quad \mathrm{otherwise}, \end{cases}$$ where $\mathcal A$ stands for domain of $L_A.$ We call the CQMS $(A,L_A)$ the base CQMS. 
Throughout this paper, our {\bf standing assumption} on a $C^{\ast}$-dynamical system $(A,\rho,\Gamma)$ is that $A$ is a unital $C^{\ast}$-algebra equipped with a Lip-norm $L_A$ and $\Gamma$ is a countable discrete group with a length function $\ell$ such that its growth function satisfies $\lambda(n)\leq e^{Cn^{\alpha}}$ for some constants $C>0,0<\alpha\leq 1$. We say that the action $\rho$ is {\bf Lip-isometric} if $L_A(\rho_t(a))=L_A(a)$ for all $t\in\Gamma$ and all $a\in A.$ We denote the unit Lip-ball with respect to the Lip-norm $L_{\ell}^{S, \beta, \gamma}$ by $\mathcal L^{S, \beta, \gamma}_1,$ i.e.,
$$\mathcal L^{S, \beta, \gamma}_1 : = \left \{\mathfrak a \in A \rtimes_{r, \rho, \sigma} \Gamma\ :\ L_{\ell}^{S, \beta, \gamma} (\mathfrak {a}) \leq 1 \right \}.$$

\begin{lem}
Let $\Gamma$ be a countable discrete group acting on a unital $C^{\ast}$-algebra $A$ via $\rho.$ If $\rho$ is Lip-isometric and $L_A$ is a Lipschitz seminorm on $A,$ then $L_{\ell}^{S, \beta, \gamma}$ is a Lipschitz seminorm on $A_{\sigma} : = A \rtimes_{r, \rho, \sigma} \Gamma$ for any $\beta, \gamma \in \mathbb R.$
\end{lem}

\begin{proof}
It is easy to see that $L_{\ell}^{S, \beta, \gamma}$ has one dimensional kernel $\mathbb C 1_A$ and that the domain $C_c \left (\Gamma, \mathcal A \right )$ of $L_{\ell}^{S, \beta, \gamma}$ is dense in $A_{\sigma}$. where $\mathcal A$ denotes the domain of $L_A.$ To show that $L_{\ell}^{S, \beta, \gamma}$ is adjoint invariant, we make use of Lip-isometry of the underlying action $\rho.$ Note that for any $\mathfrak a = \sum\limits_{t \in \Gamma} \delta_t a_t \in C_c (\Gamma, \mathcal A, \sigma),$ we have $\mathfrak {a}^{\ast} = \left (\sum\limits_{t \in \Gamma} \delta_t a_t \right )^{\ast} = \sum\limits_{t \in \Gamma} \delta_{t^{-1}} \rho_t \left (a_t^{\ast} \right ) \overline {\sigma (t, t^{-1})}.$ Since the $2$-cocycle $\sigma$ takes values in $\mathbb T,$ $L_{\ell}^{\beta, \gamma}$ is already adjoint invariant. Now using the fact that $L_A$ is adjoint invariant and $\rho$ is Lip-isometric, it is easy to see that $L_A\big(\rho_t \left (a_t^{\ast} \right ) \overline {\sigma (t, t^{-1})}\big)=L_A(a_t)$ for all $t$. Therefore, by definition, for $\mathfrak{a}=\sum\limits_{t\in \Gamma}\delta_t a_{t}\in C_{c}(\Gamma,\mathcal A,\sigma)$, 
\Bea
L^{S,\beta,\gamma}_{\ell}(\mathfrak{a}^{\ast})&=&\mathrm{max}\{L^{\beta,\gamma}_{\ell}(\mathfrak{a}^{*}), L_{A}(\rho_t \left (a_t^{\ast} \right ) \overline {\sigma (t, t^{-1})})\}\\
&=& \mathrm{max}\{L^{\beta,\gamma}_{\ell}(\mathfrak{a}), L_A(a_t)\}\\
&=& L^{S,\beta,\gamma}_{\ell}(\mathfrak{a}).
\Eea 
As for the elements where $L^{S,\beta,\gamma}_{\ell}$ takes the value infinity, the adjoint invariance follows from the fact that $C_{c}(\Gamma,\mathcal A,\sigma)$ is adjoint invariant.
\end{proof}

\begin{thm}
Let $(A,\rho,\Gamma)$ be a $C^{\ast}$-dynamical system satisfying the standard assumptions; $\Gamma$ admits a normalized $2$-cocycle $\sigma : \Gamma \times \Gamma \rightarrow \mathbb{T}$. If $\rho$ is $\mathrm{Lip}$-$\mathrm {isometric}$ then for all $\beta \geq \alpha$ and $\gamma > C,$ the associated reduced twisted crossed product $C^{\ast}$-algebra $A_{\sigma} : = A \rtimes_{r, \rho, \sigma} \Gamma$ has a CQMS structure with respect to the stratified Lip-norm $L^{S, \beta, \gamma}_{\ell}$ as defined in Subsection \ref{Lip-norm defn}.
\end{thm}

\begin{proof} Let $\varepsilon > 0.$ First we get hold of some threshold $N \in \mathbb N$ such that for all $n \geq N$ and for all $\mathfrak{a} = \sum\limits_{t \in \Gamma} \delta_t a_t \in C_c (\Gamma, \mathcal A)$ with $L_{\ell}^{S, \beta, \gamma} (\mathfrak{a}) \leq 1,$ we have 
$$\left \|\sum\limits_{\ell (t) \geq n + 1} \delta_t a_t \right \|_{\mathrm {red}} < \frac {\varepsilon} {2}.$$
First note that if $L_{\ell}^{S, \beta, \gamma} (\mathfrak{a}) \leq 1,$ then for all $t \in \Gamma \setminus \{e\},$ we have 
$$\left \|a_t \right \| \leq e^{- \gamma \ell (t)^{\beta}}.$$
Now fix some $n \geq 1.$ Then we have
\Bea
\left \|\sum\limits_{\ell (t) \geq n + 1} \delta_t a_t \right \|_{\mathrm {red}} & \leq & \sum\limits_{\ell (t) \geq n + 1} \left \|a_t \right \| \\ 
& \leq & \sum\limits_{\ell (t) \geq n + 1} e^{- \gamma \ell (t)^{\beta}} \\ 
& \leq & \sum\limits_{k = n + 1}^{\infty} \sum\limits_{k - 1 < \ell (t) \leq k} e^{- \gamma k^{\beta}} \\ 
& \leq & \sum\limits_{k = n + 1}^{\infty} e^{- \gamma k^{\beta}} e^{C k^{\alpha}}.
\Eea
By Lemma \ref{threshold}, it then follows that the right hand side of the above inequality is less than $\frac {\varepsilon} {2}$ for sufficiently large $n.$ Let $N$ be this threshold.

\vspace{2mm}

Now fix some state $\psi$ of $A,$ and define a state $\tau$ on the reduced crossed product $A_{\sigma} = A \rtimes_{r, \rho, \sigma} \Gamma$ by $\tau(\mathfrak{a}) = \psi(E(\mathfrak{a})),$ where $E : A_{\sigma} \rightarrow A$ is the canonical conditional expectation. Thus, for $\mathfrak{a} = \sum\limits_{t \in \Gamma} \delta_t a_t,$ we have $\tau(\mathfrak{a}) = \psi (a_e).$

To establish the CQMS structure on $A_{\sigma},$ we must show that the set 
$$\widetilde{\mathcal{L}}_1^{S, \beta, \gamma} := \left \{\mathfrak{a} \in A_{\sigma}\ :\ L_{\ell}^{S, \beta, \gamma} (\mathfrak{a}) \leq 1 \ \mathrm{and}\ \tau(\mathfrak{a}) = 0 \right \}$$ 
is totally bounded with respect to the reduced norm. Consequently, it is enough to show that the set of finite truncations 
$$S := \left \{\mathfrak{a}=\sum\limits_{\ell (t) \leq N} \delta_t a_t\ :\ \mathfrak{a} \in \widetilde{\mathcal{L}}_1^{S, \beta, \gamma} \right \}$$ 
forms a totally bounded set. Note that for $\mathfrak{a} \in \widetilde{\mathcal{L}}_1^{S, \beta, \gamma},$ we automatically have $\psi(a_e) = \tau(\mathfrak{a}) = 0.$
Since $(A, L_A)$ is a CQMS, the set
$$\mathcal L_1' : = \left \{a \in A\ :\ L_A (a) \leq 1\ \mathrm {and}\ \psi (a) = 0 \right \}$$ 
is totally bounded in $A$ by Theorem \ref{totbouiff}. So we can get hold of $c_1, \cdots, c_p \in \mathcal L_1'$ such that 
$$\mathcal L_1' \subseteq \bigcup\limits_{i = 1}^{p} B \left (c_i, \frac {\varepsilon} {4 \lambda_N} \right ),$$ 
where $\lambda_N : = \left \lvert \left \{t \in \Gamma\ :\ \ell (t) \leq N \right \} \right \rvert.$

Furthermore, for $t \neq e,$ the condition $L_{\ell}^{S, \beta, \gamma}(\mathfrak{a}) \leq 1$ implies $\|a_t\| \leq e^{-\gamma \ell(t)^\beta} \leq 1.$ Therefore, the scalars $\psi (a_t)$ lie in the closed unit disk $D := \{z \in \mathbb{C} \ : \ |z| \leq 1\}.$ Since $D$ is compact, there exists a finite set $Z = \{z_1, \dots, z_m\} \subseteq D$ such that
$$D \subseteq \bigcup_{j=1}^m B\left(z_j, \frac{\varepsilon}{4 \lambda_N}\right).$$
Now, let $\mathfrak{a} \in \widetilde{\mathcal{L}}_1^{S, \beta, \gamma}.$ We construct an approximation for each $a_t$ with $\ell(t) \leq N.$ 
For $t = e,$ we have $L_A(a_e) \leq 1$ and $\psi (a_e) = 0.$ Thus $a_e \in \mathcal{L}_1',$ and we can choose $b_e \in \{c_1, \dots, c_p\}$ such that
$$\left \|a_e - b_e \right \| < \frac{\varepsilon}{4 \lambda_N}.$$
For $g \neq e,$ we have $L_A(a_t) \leq 1.$ Because $L_A$ vanishes on scalars, we have $L_A(a_t - \psi(a_t)1_A) = L_A(a_t) \leq 1,$ and clearly $\psi (a_t - \psi(a_t)1_A) = 0.$ Thus $a_t - \psi(a_t)1_A \in \mathcal{L}_1',$ and we can choose $b_t \in \{c_1, \cdots, c_p\}$ such that
$$\left \|a_t - \psi (a_t) 1_A - b_t \right \| < \frac {\varepsilon} {4 \lambda_N}.$$
Simultaneously, since $\psi(a_t) \in D,$ we can choose $w_t \in Z$ such that 
$$\left |\psi (a_t) - w_t \right | < \frac{\varepsilon}{4 \lambda_N}.$$
Define $y_t = b_t + w_t 1_A$ for $t \neq e,$ and $y_e = b_e.$ Notice that each $y_t$ belongs to the finite set
$$F := \{ c_i + z_j 1_A \ : \ 1 \leq i \leq p, 1 \leq j \leq m \} \cup \{c_1, \dots, c_p\}.$$ 
We then have, for $t \neq e$
\Bea
\left \| a_t - y_t \right \| & = & \left \| (a_t - \psi (a_t)1_A - b_g) + (\psi (a_t) - w_t)1_A \right \| \\
& \leq & \left \| a_t - \psi (a_t)1_A - b_t \right \| + \left | \psi (a_t) - w_t \right | \\
& < & \frac{\varepsilon}{4 \lambda_N} + \frac{\varepsilon}{4 \lambda_N} \\ & = & \frac{\varepsilon}{2 \lambda_N}.
\Eea
For $t = e,$ we already established $\|a_t - y_t\| < \frac{\varepsilon}{4 \lambda_N} < \frac {\varepsilon} {2 \lambda_N}.$
So we have
\Bea
\left \|\sum\limits_{\ell (t) \leq N} \delta_t a_t - \sum\limits_{\ell (t) \leq N} \delta_t y_t \right \|_{\mathrm{red}} & \leq & \sum\limits_{\ell (t) \leq N} \left \|a_t - y_t \right \| \\
& < & \sum\limits_{\ell (t) \leq N} \frac{\varepsilon}{2 \lambda_N} \\ & = & \frac{\varepsilon}{2}.
\Eea
This shows that the set $S$ is totally bounded with respect to the reduced norm, as required.
\end{proof}
In this paper all the $2$-cocycles on $\Gamma$ are assumed to be normalized unless mentioned otherwise. The following lemma is going to be useful later.
\begin{lem}\label{estimate1}
Let $(A,\rho,\Gamma)$ be a $C^{\ast}$-dynamical system satisfying the standard assumption; $\sigma$ be a $2$-cocycle on $\Gamma$. Consider the associated reduced $\sigma$-twisted crossed product $C^{\ast}$-algebra $A_{\sigma}$. Then for a fixed $\beta \geq \alpha$ and for any $\varepsilon > 0,$ there exists $\gamma_0 > C$ such that for all $\gamma \geq \gamma_0$ and all $\mathfrak{a} \in A_{\sigma}$ satisfying $L_{\ell}^{S, \beta, \gamma} (\mathfrak{a}) \leq 1,$ we have $\|\mathfrak{a} - E (\mathfrak{a})\| < \varepsilon,$ where $E : A_{\sigma} \rightarrow A$ is the canonical conditional expectation. The choice of $\gamma_0$ does not depend on the choice of the cocycle $\sigma$.
\end{lem}

\begin{proof}
Let $\delta : = \min \left \{\ell (t)\ : 0 < \ell (t) \leq 1 \right \}.$ Fix some $\beta \geq \alpha.$ Then for $\gamma > 2^{\alpha} C$ and $\mathfrak{a} = \sum\limits_{t \in \Gamma} \delta_t a_t \in A_{\sigma}$ with $L_{\ell}^{S, \beta, \gamma} (\mathfrak{a}) \leq 1,$ we have
\Bea
\|\mathfrak{a} - E (\mathfrak{a})\| & = & \left \|\sum\limits_{t \in \Gamma \setminus \{e\}} \delta_t a_t \right \| \\ & \leq & \sum\limits_{\ell (t) > 0} \left \|a_t \right \| \\ & \leq & \sum\limits_{\ell (t) > 0} e^{-\gamma \ell (t)^{\beta}} \\ & = & \sum\limits_{0 < \ell (t) \leq 1} e^{-\gamma \ell (t)^{\beta}} + \sum\limits_{n = 2}^{\infty} \sum\limits_{n - 1 < \ell (t) \leq n} e^{-\gamma (n - 1)^{\beta}} \\ & \leq & e^{-\gamma \delta^{\beta}} e^{C} + \sum\limits_{n = 2}^{\infty} e^{-\gamma (n - 1)^{\beta} + C n^{\alpha}} \\ & \leq & e^{-\gamma \delta^{\alpha} + C} + \sum\limits_{n = 2}^{\infty} e^{-\gamma (n - 1)^{\alpha} + C n^{\alpha}} \\ & \leq & e^{-\gamma \delta^{\alpha} + C} + \sum\limits_{n = 2}^{\infty} e^{-\left (\frac {\gamma} {2^{\alpha}} - C \right ) n^{\alpha}} \\ & = & e^{-\gamma \delta^{\alpha} + C} + \sum\limits_{n = 2}^{\infty} e^{-\gamma' n^{\alpha}}, 
\Eea
where $\gamma' : = \frac {\gamma} {2^{\alpha}} - C > 0.$ Note that $e^{-\gamma \delta^{\alpha} + C} \xrightarrow{\gamma \rightarrow \infty} 0.$ So in order to show the desired conclusion, it is enough to show that $$\sum\limits_{n = 2}^{\infty} e^{-\gamma' n^{\alpha}} \xrightarrow{\gamma' \rightarrow \infty} 0.$$
Since $x \mapsto e^{-\gamma' x^{\alpha}}$ is decreasing on $[0, \infty),$ it follows that
$$\sum\limits_{n = 2}^{\infty} e^{-\gamma' n^{\alpha}} \leq \int_{1}^{\infty} e^{-\gamma' x^{\alpha}}\ dx.$$ 
If $\alpha \geq 1,$ then 
$$\int_{1}^{\infty} e^{-\gamma' x^{\alpha}}\ dx \leq \int_{1}^{\infty} e^{-\gamma' x}\ dx = \frac {e^{-\gamma'}} {\gamma'} \xrightarrow{\gamma' \rightarrow \infty} 0.$$
If $0 < \alpha < 1,$ we apply the substitution $u = x^{\alpha}$, which gives $dx = \frac{1}{\alpha} u^{\frac{1-\alpha}{\alpha}} \, du.$ Then we have
\Bea
\int_{1}^{\infty} e^{-\gamma' x^{\alpha}}\ dx & = & \frac{1}{\alpha} \int_{1}^{\infty} u^{\frac{1-\alpha}{\alpha}} e^{-\gamma' u} \ du \\ & = & \frac{1}{\alpha} \int_{1}^{\infty} \left( u^{\frac{1-\alpha}{\alpha}} e^{-\frac{\gamma'}{2} u} \right) e^{-\frac{\gamma'}{2} u} \ du.
\Eea
Let $p = \frac{1-\alpha}{\alpha} > 0.$ We would like to have the global maximum of the function $g(u) = u^p e^{-\frac{\gamma'}{2} u}$ on $[0, \infty).$ It is easy to see that $g$ is differentiable and the equation $g' (u) = 0$ yields a unique critical point at $u = \frac{2p}{\gamma'}.$ Since $g$ is increasing on $\left [0, \frac {2 p} {\gamma'} \right ]$ and decraesing on $\left [\frac {2 p} {\gamma'}, \infty \right ),$ it follows that $g$ has the global maximum at $u = \frac {2 p} {\gamma'}.$  Evaluating $g$ at this maximum gives
\begin{equation*}
    M = g\left(\frac{2p}{\gamma'}\right) = \left( \frac{2(1-\gamma')}{e \gamma' \alpha} \right)^{\frac{1-\alpha}{\alpha}}.
\end{equation*}
Since $g(u) \leq M$ for all $u \geq 0,$ it follows that
$$\int_{1}^{\infty} \left( u^{\frac{1-\alpha}{\alpha}} e^{-\frac{\gamma'}{2} u} \right) e^{-\frac{\gamma'}{2} u} \ du \leq M \int_{1}^{\infty} e^{-\frac{\gamma'}{2} u} \ du = \frac {2M} {\gamma'} e^{-\frac {\gamma'} {2}}.$$
Thus, 
$$\int_{1}^{\infty} e^{-\gamma' x^{\alpha}}\ dx \leq \frac {2 M} {\alpha \gamma'} e^{- \frac {\gamma'} {2}} \xrightarrow{\gamma' \rightarrow \infty} 0.$$
This shows that for any $\alpha > 0,$ $$\sum\limits_{n = 2}^{\infty} e^{-\gamma' n^{\alpha}} \xrightarrow{\gamma' \rightarrow \infty} 0,$$ as required.
\end{proof}
So we have a two parameter family of Lip-norms on the operator system $A\rtimes_{r,\rho,\sigma}\Gamma$. In the next theorem we show that as the parameter $\gamma\rightarrow\infty$, for a fixed $\beta$ and a cocycle $\sigma$, the corresponding family of CQMS converges to the base CQMS $(A,L_A)$.

\begin{thm} \label{convergence qGH}
Let $\Gamma, A, L_A, \beta, \gamma, \sigma$ be as above. Then \begin{displaymath}\mathrm {dist}_{\mathrm {qGH}} \left (\left (A_{\sigma}, L_{\ell}^{S, \beta, \gamma} \right ), \left (A, L_A \right ) \right ) \xrightarrow {\gamma \rightarrow \infty} 0.\end{displaymath}
\end{thm}

\begin{proof}
Consider the Lip-norm $L_{\varepsilon}^{\gamma}$ on $A_{\sigma} \oplus A$ defined by 
$$L_{\varepsilon}^{\gamma} (\mathfrak {b}, a) : = \max \left \{L_{\ell}^{S, \beta, \gamma} (\mathfrak {b}), L_A (a), \frac {2} {\varepsilon} \left \|E (\mathfrak {b}) - a \right \| \right \},$$
where $E : A_{\sigma} \rightarrow A$ is the canonical conditional expectation. We claim that $L_{\varepsilon}^{\gamma}$ is admissible. In order to show that $(L_{\varepsilon}^{\gamma})_{\mathrm{sa}}$ induces $(L_{\ell}^{S, \beta, \gamma})_{\mathrm{sa}},$ we take an arbitrary $\mathfrak {b} \in \big(\mathrm{Dom}(L^{S,\beta,\gamma}_{\ell})\big)_{\mathrm{sa}}$ and set $a = E (\mathfrak {b}).$ Let $\mathfrak {b} = \sum\limits_{t \in \Gamma} \delta_t b_t$ for some $b_{t}\in\mathcal A$. Then $a = E (\mathfrak {b}) = b_e.$ Note that since $E$ is adjoint invariant $a \in \mathrm{Dom}(L_A)_\mathrm{sa}.$ Therefore,
\Bea
L_A (a) & = & L_A (E (\mathfrak {b})) \\ & = & L_A \left (b_e \right ) \\ & \leq & \sup\limits_{g \in \Gamma} L_A \left (b_g \right ) \\ & \leq & \max \left \{L_{\ell}^{\beta, \gamma} (\mathfrak {b}), \sup\limits_{g \in \Gamma} L_A \left (b_g \right ) \right \} \\ & = & L_{\ell}^{S, \beta, \gamma} (\mathfrak {b}).
\Eea
Also, $$\frac {2} {\varepsilon} \left \|E (\mathfrak {b}) - a \right \| = \frac {2} {\varepsilon} \left \|E (\mathfrak {b}) - E (\mathfrak {b}) \right \| = 0 \leq L_{\ell}^{S, \beta, \gamma} (\mathfrak {b}).$$
This shows that $\big(L_{\varepsilon}^{\gamma}\big)_{\mathrm{sa}}$ induces $\big(L_{\ell}^{S, \beta, \gamma}\big)_{\mathrm{sa}}.$ Next we show that $\big(L_{\varepsilon}^{\gamma}\big)_{\mathrm{sa}}$ induces $(L_A)_{\mathrm{sa}}.$
For that, we take any $a \in \mathcal {A}_{\mathrm {sa}}$ and set $\mathfrak {b} = \delta_e a \in  \mathrm{Dom}(L^{S,\beta,\gamma}_{\ell})_{\mathrm {sa}}.$ Then 
$$\frac {2} {\varepsilon} \left \|E (\mathfrak {b}) - a \right \| = \frac {2} {\varepsilon} \left \|a - a \right \| = 0 \leq L_A (a),$$ and,
\Bea
L_{\ell}^{S, \beta, \gamma} (\mathfrak {b}) & = & L_{\ell}^{S, \beta, \gamma} \left (\delta_e a \right ) \\ & = & \max \left \{L_{\ell}^{\beta, \gamma} \left (\delta_e a \right ), L_A (a) \right \} \\ & = & L_A (a).
\Eea
Thus $(L_{\varepsilon}^{\gamma})_{\mathrm{sa}}$ also induces $(L_A)_{\mathrm{sa}}.$

Now for $\mu \in S (A_{\sigma}),$ define $\nu_{\mu} \in S (A)$ by $\nu_{\mu} = \mu \rvert_A.$ Then for $(\mathfrak {b}, a) \in A_{\sigma} \oplus A$ with $L_{\varepsilon}^{\gamma} (\mathfrak {b}, a) \leq 1$ we have
\bea\label{3.2}
\left \lvert \mu (\mathfrak {b}) - \nu_{\mu} (a) \right \rvert & = & \left \lvert \mu (\mathfrak {b}) - \mu (a) \right \rvert \nonumber \\ & = & \left \lvert \mu (\mathfrak {b} - E (\mathfrak {b})) + \mu (E (\mathfrak {b})) - \mu (a) \right \rvert \nonumber \\ & \leq & \left \vert \mu (\mathfrak {b} - E (\mathfrak {b}))  \right \rvert + \left \lvert \mu (E (\mathfrak {b}) - a) \right \rvert \nonumber \\ & \leq & \|\mathfrak {b} - E (\mathfrak {b})\| + \|E (\mathfrak {b}) - a \| \nonumber \\ & \leq & \|\mathfrak {b} - E (\mathfrak {b})\| + \frac {\varepsilon} {2}. 
\eea
Similarly, for $\nu \in S (A),$ setting $\mu_{\nu} \in S \left (A_{\sigma} \right )$ to be a Hahn-Banach extension of $\nu$ to $A$ and for all $(\mathfrak {b}, a) \in B \oplus A$ with $L_{\varepsilon}^{\gamma} (\mathfrak {b} , a) \leq 1,$ we have
\begin{equation}\label{3.3} \left \lvert \mu_{\nu} (\mathfrak {b}) - \nu (a) \right \rvert \leq \|\mathfrak {b} - E (\mathfrak {b})\| + \frac {\varepsilon} {2}. \end{equation}
As $L^{\gamma}_{\varepsilon}(\mathfrak{b},a)\leq 1$ implies $L^{S,\beta,\gamma}_{\ell}(\mathfrak{b})\leq 1$,  by Lemma \ref{estimate1}, there is $\gamma_0$ such that for all $\gamma\geq\gamma_0$ and $L^{\gamma}_{\varepsilon}(\mathfrak{b},a)\leq 1$, $\lvert\lvert E(\mathfrak{b})-\mathfrak{b}\rvert\rvert<\frac{\varepsilon}{2}$. Therefore, by Equations \ref{3.2} and \ref{3.3}, it follows that $$\mathrm{dist}_{\mathrm {GH}}^{L_{\varepsilon}^{\gamma}} \left (S \left (A_{\sigma} \right ), S (A) \right ) < \varepsilon,$$ for all $\gamma\geq\gamma_0$ and consequently, $$\mathrm{dist}_{\mathrm{qGH}} \left (\left (A_{\sigma}, L_{\ell}^{S, \beta, \gamma} \right ), \left (A, L_A \right ) \right ) \xrightarrow{\gamma \rightarrow \infty} 0.$$

\end{proof}


Now we are going to introduce a third parameter for our CQMS family. The parameters $\beta,\gamma$ parametrize Lip-norms. The third parameter would parametrize the underlying operator systems. To that end, let $(A,\rho,\Gamma)$ be a $C^{\ast}$-dynamical satisfying the standard assumptions such that $\rho$ is Lip-isometric and let $\Gamma$ admit a strongly continuous one-parameter family $\{\sigma_{\theta}\}_{\theta\in\Omega}$ of $2$-cocycles where $\Omega$ is some compact parameter space. Then by  Theorem \ref{upper semicontinuity}, $\{A\rtimes_{r,\rho,\sigma_{\theta}}\Gamma\}_{\theta\in\Omega}$ is a continuous family of $C^{\ast}$-algebras provided $\Gamma$ is exact. For the rest of this section the group $\Gamma$ is assumed to be {\bf exact}. Then we have a three parameter family of CQMS given by $\{\big(A\rtimes_{r,\rho,\sigma_{\theta}}\Gamma,L^{S,\beta,\gamma}_{\ell}\big)\}_{\theta,\beta,\gamma}$ where $\beta\ge \alpha$, $\gamma>C$ and $\theta\in\Omega$. We shall prove that this family is jointly continuous in the parameters with respect to the quantum Gromov-Hausdorff distance. We shall adapt the techniques used by M. Rieffel used in \cite{Rieffel-Gromov}*{Theorem 9.2} to prove continuity of quantum tori with respect to the deformation parameter. So the first thing we seek for is a suitable finite dimensional approximation. To that end, fix a point $(\beta_0,\gamma_0,\theta_0)$ in the parameter space $[\alpha,+\infty)\times(C,+\infty)\times\Omega$. For any positive real number $x$, let us define for a fixed $\theta\in \Omega$,
\begin{displaymath}
    A_{\theta}^{x} : = \left \{\sum\limits_{\ell (t) \leq x} \delta_t a_t\ :\ a_t \in A \right \} \subseteq C_c \left (\Gamma, A, \sigma_{\theta} \right ).
\end{displaymath}  
Then it is easy to see that $A_{\theta}^{x}$ is a closed and hence complete operator system and consequently $(A_{\theta}^{x},L^{S,\beta,\gamma}_{\ell}\big\rvert_{A^{x}_{\theta}})$ is again a CQMS. For brevity, from now on, we shall denote the $C^{\ast}$-algebra $A\rtimes_{r,\rho,\sigma_{\theta}}\Gamma$ by $A_{\theta}$.

\begin{thm} \label{tail threshold}
Let $\eta=\frac{\gamma_0-C}{2}$. Retaining all the notations introduced earlier, for an $\varepsilon>0$, for all $\beta\ge\alpha$, $\gamma\ge\gamma_0 - \eta$ and all $\theta\in\Omega$, there is an $R>0$ such that
\begin{displaymath}
 \mathrm {dist}_{\mathrm {qGH}} \left ( A_{\theta}, L_{\ell}^{S, \beta, \gamma} \right ), \left (A_{\theta}^{R}, L_{\ell}^{S, \beta, \gamma} \big \rvert_{A_{\theta}^{R}} \right ) < \varepsilon.   
\end{displaymath}
\end{thm}

\begin{proof}
Fix any $\theta \in \Omega$ arbitrarily. Then note that for all $\theta \in \Omega,$ $x > 0$ and $\mathfrak {a} = \sum\limits_{t \in \Gamma} \delta_t a_t \in C_c \left (\Gamma, \mathcal A, \sigma_{\theta} \right )$ we have 
\bea
\left \|\sum\limits_{\ell (t) > x} \delta_t a_t \right \|_{\mathrm {red}, \sigma_{\theta}} & \leq & \sum\limits_{\ell (t) > x} \left \|a_t \right \| \nonumber \\ & \leq & L_{\ell}^{S, \beta, \gamma} (\mathfrak {a}) \sum\limits_{\ell (t) > x} e^{-\gamma \ell (t)^{\beta}} \nonumber \\ & \leq & L_{\ell}^{S, \beta, \gamma} (\mathfrak {a}) \sum\limits_{\ell (t) > \left \lfloor x \right \rfloor} e^{- \gamma \ell (t)^{\beta}} \nonumber \\ & = & L_{\ell}^{S. \beta, \gamma} (\mathfrak {a}) \sum\limits_{n = \left \lfloor x \right \rfloor}^{\infty} \sum\limits_{n < \ell (t) \leq n + 1} e^{-\gamma n^{\beta}} \nonumber \\ & \leq & \label{tail estimate} L_{\ell}^{S, \beta, \gamma} (\mathfrak {a}) \sum\limits_{n = \left \lfloor x \right \rfloor}^{\infty} e^{-\gamma n^{\beta} + C n^{\alpha}}. \label{tail estimation}
\eea
By virtue of \ref{threshold} and Remark \ref{gamma independence}, the final tailed sum can be made smaller than any positive $\varepsilon$ by choosing $x$ sufficiently large, for all $\beta\geq\alpha,\gamma\geq\eta$ and  all $\theta \in \Omega.$ Let the threshold thus obtained be $R.$ We claim that
$$\mathrm {dist}_{\mathrm {GH}}^{L_{\varepsilon}^{R}} \left (\left (A_{\theta}, L_{\ell}^{S, \beta, \gamma} \right ), \left (A_{\theta}^{R}, L_{\ell}^{S, \beta, \gamma} \big \rvert_{A_{\theta}^{R}} \right ) \right ) < \varepsilon,$$
where $L_{\varepsilon}^{R}$ is an admissible Lip-norm on $A_{\theta} \oplus A_{\theta}^{R}$ defined by
$$L_{\varepsilon}^{R} (\mathfrak {a}, \mathfrak {b}) = \max \left \{L_{\ell}^{S, \beta, \gamma} (\mathfrak {a}), L_{\ell}^{S, \beta, \gamma} \big \rvert_{A_{\theta}^{R}} (\mathfrak {b}), \frac {1} {\varepsilon} \|\mathfrak {a} - \mathfrak {b}\| \right \},$$ for all $\mathfrak {a} \in A_{\theta}$ and $\mathfrak {b} \in A_{\theta}^{R}.$ 

The fact that $L_{\varepsilon}^{R}$ induces  $L_{\ell}^{S, \beta, \gamma} \big \rvert_{A_{\theta}^{R}}$ on the self-adjoint part is a triviality. So it is enough to show that $L_{\varepsilon}^{R}$ induces $L_{\ell}^{S, \beta, \gamma}.$ For that, fix some $\mathfrak{a}=\mathfrak{a}^{\ast} \in A_{\theta}.$ Without loss of generality we may assume that $\mathfrak {a} = \sum\limits_{t \in \Gamma} \delta_t a_t \in C_c \left (\Gamma, \mathcal A, \sigma_{\theta} \right ).$ Set $\mathfrak {b} = \mathfrak {a}_R = \sum\limits_{\ell (t) \leq R} \delta_t a_t.$ Then $\mathfrak{b}\in (A_{R}^{\theta})_{\mathrm{sa}}$. clearly we have
$$L_{\ell}^{S, \beta, \gamma} \big \rvert_{A_{\theta}^{R}} (\mathfrak {b}) \leq L_{\ell}^{S, \beta, \gamma} (\mathfrak {a}).$$
Also, by the equation \eqref{tail estimate}, we have 
$$\|\mathfrak {a} - \mathfrak {b}\| < \varepsilon L_{\ell}^{S, \beta, \gamma} (\mathfrak {a}).$$
Thus $L_{\varepsilon}^{R}$ induces $L_{\ell}^{S, \beta, \gamma},$ on the self-adjoint part as required.

Let $d_{L_{\varepsilon}^{R}}$ be the associated metric on $S \left (A_{\theta} \oplus A_{\theta}^{R} \right ).$ Now for $\mu \in S \left (A_{\theta} \right ),$ setting $\nu_{\mu} = \mu \rvert_{A_{\theta}^{R}} \in S \left (A_{\theta}^{R} \right ),$ we find that for all $\mathfrak {a} \in A_{\theta}$ and $\mathfrak {b} \in A_{\theta}^{R}$ with $L_{\varepsilon}^{R} (\mathfrak {a}, \mathfrak {b}) \leq 1$
$$\left \lvert \mu (\mathfrak {a}) - \nu_{\mu} (\mathfrak {b}) \right \rvert = \left \lvert \mu (\mathfrak {a}) - \mu (\mathfrak {b}) \right \rvert \leq \|\mathfrak {a} - \mathfrak {b}\| \leq \varepsilon L_{\varepsilon}^{R} (\mathfrak {a}, \mathfrak {b}) \leq \varepsilon.$$ In other words, $\mu \in V_{\varepsilon} (\nu_{\mu}).$ Since this holds for any $\mu \in S \left (A_{\theta} \right ),$ it follows that $S \left (A_{\theta} \right ) \subseteq V_{\varepsilon} \left (S \left (A_{\theta}^{R} \right ) \right ).$ Conversely, for any $\nu \in S \left (A_{\theta}^{R} \right ),$ setting $\mu_{\nu}$ to be any Hahn-Banach extension of $\nu$ to $A_{\theta},$ it follows that $\nu \in V_{\varepsilon} \left (\mu_{\nu} \right ).$ Since this also holds for any $\nu \in S \left (A_{\theta}^{R} \right ),$ we conclude that $S \left (A_{\theta}^{R} \right ) \subseteq V_{\varepsilon} \left ( S \left (A_{\theta} \right ) \right ).$ This shows that
$$\mathrm {dist}_{\mathrm {qGH}} \left (\left (A_{\theta}, L_{\ell}^{S, \beta, \gamma} \right ), \left (A_{\theta}^{R}, L_{\ell}^{S, \beta, \gamma} \big \rvert_{A_{\theta}^{R}} \right ) \right ) < \varepsilon,$$
as required.

\end{proof}

To obtain a finite dimensional approximation, we are going to put a natural assumption on the base CQMS $(A,L_A).$ 

\begin{defn} \label{fdim approx prop}

We say that the base CQMS admits the {\bf finite dimensional approximation property} if there is a sequence $\left \{A_n \right \}_{n \geq 1}$ of self-adjoint finite dimensional subspaces of $A$ containing the identity and a sequence of continuous maps $E_n : A \rightarrow A_n$ such that for all $n \geq 1$ and for all $a \in A,$ $$L_A \left (E_n (a) \right ) \leq L_A (a),$$ and, $$\left \|a - E_n (a) \right \| < \beta (n) L_A (a),$$ where $\{\beta (n)\}_{n \geq 1}$ is a sequence of positive real numbers converging to zero.\end{defn} 
This type of approximation property is satisfied by AF algebras, continuous functions on classical compact space admitting ergodic action of a compact Lie group which has a faithful, finite dimensional unitary representations. For $n\in\mathbb{N}$ and $x>0$, consider the finite dimensional complete operator system \begin{displaymath}A_{\theta, n}^{x} : = \left \{\sum\limits_{\ell (g) \leq R} \delta_g b_g\ :\ b_g \in A_n \right \} \subseteq A_{\theta}^{x}. \end{displaymath}
Then clearly $\left (A_{\theta,n}^{x},L_{\ell}^{S,\beta,\gamma}\big\vert_{A_{\theta,n}^{x}} \right )$ is a CQMS.

\begin{thm} \label{fin-dim threshold}
Let $(A,\rho,\Gamma)$ be a $C^{\ast}$-dynamical system satisfying the standard assumption; $(A,L_A)$ is a CQMS satisfying the finite dimensional approximation property; $\{\sigma_{\theta}\}_{\theta\in\Omega}$ be a strongly continuous one-parameter family of unitary $2$-cocyles on $\Gamma$. Then for any $\varepsilon>0$, there is $n_{0}\in\mathbb{N}$ depending on $x$ and independent of the parameters $\beta,\gamma,\theta$ such that 
for all $n\ge n_0$, \begin{displaymath}\mathrm {dist}_{\mathrm {qGH}} \left (\left (A_{\theta}^{x}, L_{\ell}^{S, \beta, \gamma} \big \rvert_{A_{\theta}^{R}} \right ), \left (A_{\theta, n}^{x}, L_{\ell}^{S, \beta, \gamma} \big \rvert_{A_{\theta, n}^{R}} \right ) \right )<\varepsilon.\end{displaymath}
\end{thm}

\begin{proof}
Fix some $\theta,\beta,\gamma$ and $\varepsilon > 0$ arbitrarily. For each $n \in \mathbb N$ and $x > 0,$ define a Lip-norm $L_{\varepsilon}^{R, n}$ on $A_{\theta}^{x} \oplus A_{\theta}^{x, n}$ by
$$L_{\varepsilon}^{x, n} (\mathfrak {a}, \mathfrak {b}) = \max \left \{L_{\ell}^{S, \beta, \gamma} \big \rvert_{A_{\theta}^{x}} (\mathfrak {a}), L_{\ell}^{S, \beta, \gamma} \big \rvert_{A_{\theta, n}^{x}} (\mathfrak {b}), \frac {1} {\varepsilon} \|\mathfrak {a} - \mathfrak {b}\| \right \}.$$
We claim that $L_{\varepsilon}^{x, n}$ is admissible. Again, the fact that $L_{\varepsilon}^{x, n}$ induces $ L_{\ell}^{S, \beta, \gamma} \big \rvert_{A_{\theta, n}^{x}}$ on the self-adjoint part is a triviality. So to show the admissibility of $L_{\varepsilon}^{x, n},$ it is enough to demonstrate that $L_{\varepsilon}^{x, n}$ induces $L_{\ell}^{S, \beta, \gamma} \big \rvert_{A_{\theta}^{x}}.$ on the self-adjoint part. For that, we take any $\mathfrak {a} = \sum\limits_{\ell (t) \leq x} \delta_t a_t \in (A_{\theta}^{x})_{\mathrm{sa}}$ and set $\mathfrak {b} = \sum\limits_{\ell (t) \leq x} \delta_t E_n \left (a_t \right ) \in (A_{\theta, n}^{x})_{\mathrm{sa}}.$ Then
\Bea
L_{\ell}^{S, \beta, \gamma} \big \rvert_{A_{\theta, n}^{x}} (\mathfrak {b}) & = & \max \left \{L_{\ell}^{\beta, \gamma} (\mathfrak {b}), \sup\limits_{\ell (t) \leq x} L_A \left (E_n \left (a_t \right ) \right ) \right \} \\ & \leq & \max \left \{\left (\sum\limits_{0 < \ell (t) \leq x} e^{2 \gamma \ell (t)^{\beta}} \left \|E_n \left (a_t \right ) \right \|^{2} \right )^{\frac {1} {2}}, \sup\limits_{\ell (t) \leq x} L_A \left (a_t \right ) \right \} \\ & \leq & \max \left \{\left (\sum\limits_{0 <\ell (t) \leq x} e^{2 \gamma \ell (t)^{\beta}} \left \|a_t \right \|^{2} \right )^{\frac {1} {2}}, \sup\limits_{\ell (t) \leq x} L_A \left (a_t \right ) \right \} \\ & = & L_{\ell}^{S, \beta, \gamma} \big \rvert_{A_{\theta}^{x}} (\mathfrak {a}).  
\Eea
Also,
\Bea
\|\mathfrak {a} - \mathfrak {b}\| & = & \left \|\sum\limits_{\ell (t) \leq x} \delta_t \left (a_t - E_n \left (a_t \right ) \right ) \right \| \\ & \leq & \sum\limits_{\ell (t) \leq x} \left \|a_t - E_n \left (a_t \right ) \right \| \\ & \leq & \beta (n) \gamma_x \sup\limits_{\ell (t) \leq x} L_A \left (a_t \right ) \\ & \leq & \beta(n) \gamma_x\ L_{\ell}^{S, \beta, \gamma} \big \rvert_{A_{\theta}^{x}} (\mathfrak {a}), 
\Eea
where $\gamma_x : = \left \lvert \left \{t \in \Gamma\ : \ell (t) \leq x \right \} \right \rvert.$ Now since $\beta (n) \xrightarrow{n \rightarrow \infty} 0,$ for sufficiently large $n,$ $\beta (n) < \frac {\varepsilon} {\gamma_x}.$ Clearly such an $n$ is independent of $\theta,\beta,\gamma.$ Consequently, for large enough $n \in \mathbb N,$ we have
    $$\|\mathfrak {a} - \mathfrak {b}\| \leq L_{\ell}^{S, \beta, \gamma} \big \rvert_{A_{\theta}^{x}} (\mathfrak {a}).$$ This shows that $L_{\varepsilon, n}^{x}$ induces $L_{\ell}^{S, \beta, \gamma} \big \rvert_{A_{\theta}^{x}}$ on the self adjoint part, as required. Then adapting the argument given in the proof of the last theorem, one can finish the proof. 
\end{proof}
 In order to apply Rieffel's argument from \cite{Rieffel-Gromov}, we consider the order unit space which is the self adjoint part of the finite dimensional vector space $A^{x}_{n,\theta}=\{\sum\limits_{\ell(t)\leq x}\delta_t a_t:a_t\in\mathcal A\}$. Then, as observed earlier, (Proposition \ref{operatorsystemtoorder}) the pair $\big((A^{x}_{n,\theta})_{\mathrm{sa}},L^{S,\beta,\gamma}_{\ell}\vert_{(A^{x}_{n,\theta})_{\mathrm{sa}}}\big)$ is a CQMS in the sense of Rieffel. We shall denote the order unit space $(A^{x}_{n,\theta})_{\mathrm{sa}}$ by $V_{\theta}$ which is equipped with the order unit norm $ \|\cdot\|_{\theta}$. We denote the common underlying finite dimensional vector space of every $V_{\theta}$ by $V$. Also we denote the Lip-norm $L^{S,\beta,\gamma}_{\ell}\vert_{(A^{x}_{n,\theta})_{\mathrm{sa}}}$ simply by $L_{\beta,\gamma}$. To state the next theorem we need to recall the following definition from \cite{Rieffel-Gromov}*{Definition 10.8}:
\begin{defn}\label{contfamilystates}
Let $V$ be an order unit space equipped with a continuous field of order unit norms $\left \{\|\cdot\|_{\theta} \right \}_{\theta \in \Omega},$ varying over some compact subset $\Omega \subseteq \mathbb R.$ For the order unit space $V_{\theta} = \left (V, \|\cdot\|_{\theta} \right ),$ let $S (V_{\theta})$ be the state space of $V_{\theta}.$ By a continuous field of states, we mean a function $\Phi : \Omega \rightarrow V',$ such that $\Phi_{\theta} \in S (V_{\theta})$ for each $\theta \in \Omega,$ and such that $\theta \mapsto \Phi_{\theta} (v)$ is continuous for each $v \in V.$       
\end{defn}

\begin{thm}
As before let $(\beta_0, \gamma_0, \theta_0)$ be a fixed point in the parameter space and $x > 0;$ $\mathcal S$ be a finite, non-empty set of continuous field of states on $V_{\theta}$. For a given $\varepsilon > 0$, $\beta \geq \alpha,\gamma > C$ and a cocycle parameter $\theta$, define a seminorm $N_{\beta, \theta, \gamma, \varepsilon}$ on $V_{\theta_0} \oplus V_{\theta}$ by
$$N_{\beta, \theta, \gamma, \varepsilon} (u, v) : = \frac {1} {\varepsilon} \max \left \{\left\lvert \Phi_{\theta} (u) - \Phi_{\theta_0} (v) \right\rvert\ :\ \Phi \in \mathcal S \right\}.$$
Then there exists $\delta > 0$ such that for all $\beta \in \left [\beta_0 - \delta, \beta_0 + \delta \right ] \cap [\alpha, \infty),$ for all $\gamma \in \left [\gamma_0 - \delta, \gamma_0 + \delta \right ]$ and for all $\theta$ in a small neighborhood of $\theta_0$, $N_{\beta, \theta, \gamma, \varepsilon}$ is a bridge between $\left (V_{\theta_0}, L_{\beta_0, \gamma_0} \right )$ and $\left (V_{\theta}, L_{\beta, \gamma} \right ).$  In other words, the Lipschitz seminorm $L_{\beta_0, \theta_0, \gamma_0, \varepsilon}^{\beta,\theta,\gamma}$ given by $$L_{\beta_0, \theta_0, \gamma_0, \varepsilon}^{\beta,\theta,\gamma} (u, v) : = \max \left\{L_{\beta_0, \gamma_0} (u), L_{\beta, \gamma} (v), N_{\beta,\theta, \gamma, \varepsilon} (u, v) \right\}$$ is admissible which induces both $L_{\beta, \gamma}$ and $L_{\beta_0, \gamma_0},$ provided that $\beta \in \left [\beta_0 - \delta, \beta_0 + \delta \right ] \cap [\alpha, \infty),$ $\gamma \in \left [\gamma_0 - \delta, \gamma_0 + \delta \right ]$ and $\theta$ is in a sufficiently small neighborhood around $\theta_0$.
\end{thm}

\begin{proof}
Let us choose $\varepsilon > 0$ arbitrarily. We shall show that there is a $\delta>0$ and a neighborhood $U$ of $\theta_0$ such that for all $\beta\in[\beta_0-\delta,\beta_0+\delta] \cap [\alpha, \infty)$ and $\gamma\in[\gamma_0-\delta,\gamma_0+\delta],$  $L_{\beta_0, \theta_0, \gamma_0, \varepsilon}^{\beta,\theta,\gamma}$ induces $L_{\beta, \gamma}$ for all $\theta\in U$. It suffices to show that there is a neighborhood $U$ of $\theta_0$ such that given an element $u \in V,$ there exists $v \in V$ such that $L_{\beta_0, \theta_0, \gamma_0, \varepsilon}^{\beta,\theta,\gamma} (u, v) \leq L_{\beta,\gamma} (u)$ for all $(\beta,\gamma,\theta)\in \left ([\beta_0-\delta,\beta_0+\delta] \cap [\alpha, \infty) \right ) \times [\gamma_0-\delta,\gamma_0+\delta]\times U$. Note that $\delta_e$ is the multiplicative identity in $V.$ If $u \in \mathbb R \delta_e,$ then $v = u$ does the job for us. So without loss of generality we may assume that $u \notin \mathbb R \delta_e.$ Let $W$ be the subspace of $V$ complimentary to $\mathbb R \delta_e$. For a fixed vector $v\in V$, by Corollary \ref{continuousfielCalgebras}, the map $\theta\mapsto\lvert\lvert v\rvert\rvert_{\theta}$ is a continuous function. Therefore, by \cite{Rieffel-Gromov}*{Lemma 10.1}, we get a norm $\|\cdot\|_{\ast}$ on $V$ such that $\lvert\lvert v\rvert\rvert_{\theta}\leq \lvert\lvert v\rvert\rvert_{\ast}$ for all $v\in V$ and all $\theta \in \Omega.$ Let $\Sigma_W$ be the unit sphere of $W$ corresponding to the norm $\|\cdot\|_{\ast}$. Note that each $L_{\beta,\gamma}$ is a norm on $W.$ It can be easily checked that $L_{\beta,\gamma}'$s are continuous field of Lip-norms on $V.$ Thus it follows that $L_{\beta,\gamma}$'s form a continuous field of norms on $W.$ By similar argument as in Lemma 10.1 of Rieffel, the function $(\beta, \gamma, w) \mapsto L_{\beta,\gamma} (w)$ is jointly continuous on $\left (\left [\beta_0 - \delta_0, \beta_0 + \delta_0 \right ] \cap [\alpha, \infty) \right ) \times \left [\gamma_0 - \delta_0, \gamma_0 + \delta_0 \right ] \times W,$ for some $\delta_0>0$. So there exists $\lambda_0 > 0$ such that $L_{\beta, \gamma} (w) \geq \lambda_0$ for all $\beta \in \left [\beta_0 - \delta_0, \beta_0 + \delta_0 \right ] \cap [\alpha, \infty),$ $\gamma \in \left [\gamma_0 - \delta_0, \gamma_0 + \delta_0 \right ]$ and $w \in \Sigma_W.$ By joint continuity and compactness, there exists some $0 < \delta_1 < \delta_0$ such that for all $\beta \in \left [\beta_0 - \delta_1, \beta_0 + \delta_1 \right ] \cap [\alpha, \infty),$ $\gamma \in \left [\gamma_0 - \delta_1, \gamma_0 + \delta_1 \right ]$ and $w \in \Sigma_W$ we have $$\left \lvert L_{\beta, \gamma} (w) - L_{\beta_0, \gamma_0} (w) \right \rvert < \frac {\varepsilon \lambda_0^{2}} {2}.$$ Since $\mathcal S$ is finite, by pointwise uniform continuity of $\Phi$ and compactness of $\Sigma_W,$ it follows that there exists $0 < \delta \leq \delta_1$ such that for all $\Phi \in \mathcal S$ and for all $w \in \Sigma_W,$ 
$$\left \lvert \Phi_{\theta} (w) - \Phi_{\theta_0} (w) \right \rvert < \frac {\varepsilon \lambda_0} {2},$$ whenever $\theta$ is in a suitable small compact neighborhood (say $U$) of $\theta_0$ contained in $\Omega.$ Now we need to show that for any $u \in V \setminus \mathbb R e$ and for any $\theta\in U$, there exists some $v \in V$ such that $$L_{\beta_0, \theta_0, \gamma_0, \varepsilon}^{\beta,\theta,\gamma} (u, v) \leq L_{\beta,\gamma} (u),$$ for all $\beta$ and $\gamma$ in some appropriate neighbourhood of $\beta_0$ and $\gamma_0$ respectively. First let us take $u = w \in \Sigma_W.$ For $\beta \in \left [\beta_0 - \delta, \beta_0 + \delta \right ] \cap [\alpha, \infty],$ $\gamma \in \left [\gamma_0 - \delta, \gamma_0 + \delta \right ],$ choose $z = \frac {L_{\beta, \gamma} (w)} {L_{\beta_0, \gamma_0} (w)} w$ so that $L_{\beta_0, \gamma_0} (z) = L_{\beta, \gamma} (w).$ Then for any $\theta\in U$ and for any $\Phi \in \mathcal S$. we have
\Bea
\left \lvert \Phi_{\theta} (w) - \Phi_{\theta_0} (z) \right \rvert & \leq & \left \lvert \Phi_{\theta} (w) - \Phi_{\theta_0} (w) \right \rvert + \left \lvert \Phi_{\theta_0} (w) - \Phi_{\theta_0} (z) \right \rvert \\ & \leq & \frac {\varepsilon \lambda_0} {2} + \left \lvert 1 - \frac {L_{\beta, \gamma} (w)} {L_{\beta_0, \gamma_0} (w)} \right \rvert \left \lvert \Phi_{\theta} (w) \right \rvert \\ & \leq & \frac {\varepsilon \lambda_0} {2} + \frac {\varepsilon \lambda_0^{2}} {2} \left \lvert L_{\beta_0, \gamma_0} (w) \right \rvert^{-1} \|w\|_{\ast} \\ & \leq & \frac {\varepsilon \lambda_0} {2} + \frac {\varepsilon \lambda_0^{2}} {2} \lambda_0^{-1} \\ & = & \varepsilon \lambda_0.     
\Eea
This shows that $N_{\beta,\theta, \gamma, \varepsilon} (w, z) \leq \lambda_0 \leq L_{\beta, \gamma} (w).$ Now for $u = w + \eta \delta_e$ with $w \in \Sigma_W$ and $\eta \in \mathbb R,$ we set $v = z + \eta \delta_e.$ Then we have $L_{\beta_0, \gamma_0} (v) =L_{\beta, \gamma} (u)$ and $\left \lvert \Phi_{\theta_0} (u) - \Phi_{\theta} (v) \right \rvert \leq \varepsilon \lambda_0$ for any $\Phi \in \mathcal S.$ Since every element of $V \setminus \mathbb R \delta_e$ is a positive scalar multiple of elements of the form $w + \eta \delta_e$ with $\eta \in \mathbb R$ and $w \in \Sigma_W,$ it follows that $L_{\beta_0, \theta_0, \gamma_0, \varepsilon}^{\beta,\theta,\gamma}$ induces $L_{\beta, \gamma}$ for any $\theta\in U$ and $(\beta,\gamma) \in \left (\left [\beta_0 - \delta, \beta_0 + \delta \right ] \cap [\alpha, \infty) \right ) \times [\gamma_0-\delta,\gamma_0+\delta]$ and hence it induces $L_{\beta_0, \gamma_0}$ as well, by symmetry. In other words, $L_{\beta_0, \theta_0, \gamma_0, \varepsilon}^{\beta,\theta,\gamma}$ is an admissible Lip-norm for all $\beta \in \left [\beta_0 - \delta, \beta_0 + \delta \right ] \cap [\alpha, \infty),$ $\gamma \in \left [\gamma_0 - \delta, \gamma_0 + \delta \right ]$ and $\theta\in U$. 

\end{proof}

In the next theorem, we apply \cite{Rieffel-Gromov}*{Theorem 10.13}. In our setting, the norms $\|\cdot\|_{\theta}$ are order unit norms on the common vector space $V$. Therefore, by a well-known result of Ellis \cite{Erik-Ellis}*{Theorem II.1.15}, their duals ${\|\cdot\|_{\theta}}^{\prime}$ form a continuous field of base norms on $V^{\prime}$. Invoking \cite{Rieffel-Gromov}*{Theorem 10.13}, for any norm ${\|\cdot\|_{\ast}}^{\prime}$ on $V^{\prime}$ and any $\varepsilon > 0$, there exists a finite continuous field of states $\mathcal {S}$ which is $\varepsilon$-dense in $S(V_{\theta})$ with respect to the norm ${\|\cdot\|_{\ast}}^{\prime}$ for each $\theta \in \Omega.$ The proof of the following theorem is a straightforward adaptation of the proof of \cite{Rieffel-Gromov}*{Theorem 11.2}. However, we have decided to keep it in full details to make the exposition self-contained, particularly because our framework requires dealing with multiple parameters.

\begin{thm} \label{joint continuity lem}
With all the previous notations, for a given $\varepsilon > 0$ there exists $\delta > 0$ and a neighborhood $U$ around $\theta_0$ such that for all $(\beta,\gamma,\theta) \in \left [\beta_0 - \delta, \beta_0 + \delta \right ] \cap [\alpha, \infty)\times [\gamma_0 - \delta, \gamma_0 + \delta ]\times U$, we have
$$\mathrm{dist}_{Q} \left (\left (V_{\theta}, L_{\beta, \gamma} \big \rvert_{V_{\theta}} \right ), \left (V_{\theta_0}, L_{\beta_0, \gamma_0} \big \rvert_{V_{\theta_0}} \right ) \right ) < \varepsilon.$$
\end{thm}

\begin{proof}
Let $\tilde V = V/ \mathbb R e.$ Then each $L_{\beta, \gamma}$ becomes a norm $\widetilde {L_{\beta, \gamma}}$ on $\tilde V.$ Since $\left \{L_{\beta, \gamma} \right \}_{\beta \geq \alpha, \gamma > C}$ is a continuous field of Lip-norms on $V,$ it follows that $\left \{\widetilde {L_{\beta, \gamma}} \right \}_{\beta \geq \alpha, \gamma > C}$ is a continuous field of norms on $\tilde {V}.$ According to Lemma 10.1 of Rieffel, the field of dual norms $\left \{\widetilde {L_{\beta, \gamma}}^{\prime} \right \}_{\gamma > C}$ is also continuous on $\tilde {V}^{\prime}.$ Note that the dual of $\tilde {V}$ is canonically identified with the subspace ${V^{\prime}}^{\circ}$ of $V^{\prime}$ annihilating the one dimensional subspace $\mathbb R e$ of $V.$ Thus each $\widetilde {L_{\beta, \gamma}}^{\prime}$ give rise to a norm $L_{\beta, \gamma}^\prime$ on $V^{\prime \circ}$ and the field of norms $\left \{L_{\beta, \gamma}^{\prime} \right \}_{\beta \geq \alpha, \gamma > C}$ on $V^{\prime \circ},$ thus obtained, are also continuous. Again by virtue of \cite{Rieffel-Gromov}*{Lemma 10.1}, there exists a norm $\|\cdot\|_{\ast}^{\prime}$ on ${V^{\prime}}^{\circ}$ such that $L_{\beta, \gamma}^{\prime} \leq \|\cdot\|_{\ast}^{\prime}$ for all $\beta \in \left [\beta_0 - \delta_0, \beta_0 + \delta_0 \right ] \cap [\alpha, \infty)$ and $\gamma \in \left [\gamma_0 - \delta_0, \gamma_0 + \delta_0 \right ],$ for some $\delta_0>0$. Then for all $\theta \in \Omega$ for all $\mu, \nu \in S \left (V_{\theta} \right ),$ for all $\beta \in \left [\beta_0 - \delta_0, \beta_0 + \delta_0 \right ] \cap [\alpha, \infty)$ and $\gamma \in \left [\gamma_0 - \delta_0, \gamma_0 + \delta_0 \right ],$ we have 
$$\rho_{L_{\beta, \gamma}} (\mu, \nu) \leq L_{\beta, \gamma}^{\prime} (\mu - \nu) \leq \|\mu - \nu \|_{\ast}^{\prime}.$$
Let $\varepsilon > 0$ be given. According to \cite{Rieffel-Gromov}*{Theorem 10.11}, we can find a finite family $\mathcal S$ of continuous field of states such that $\left \{\Phi_{\theta}\ :\ \Phi \in \mathcal S \right \}$ is $\frac {\varepsilon} {2}$-dense in $S \left (V_{\theta} \right )$ with respect to the norm $\|\cdot\|_{\ast}^{\prime}$ for every $\theta \in \Omega.$ Define $N_{\beta, \theta, \gamma, \varepsilon}$ as in the last Theorem with $\varepsilon$ replaced by $\frac {\varepsilon} {2}.$ We denote the bridge by $N$. By the last theorem we can choose $0 < \delta < \delta_0$ such that $N$ is a bridge between $\left (V_{\theta}, L_{\beta, \gamma} \right )$ and $\left (V_{\theta_0}, L_{\beta_0, \gamma_0} \right ),$ for all $\beta \in \left [\beta_0 - \delta, \beta_0 + \delta \right ] \cap [\alpha, \infty),$ $\gamma \in \left [\gamma_0 - \delta, \gamma_0 + \delta \right ]$ and for any $\theta\in U$. We show that this $\delta$ and $U$ do the job for us.

For $\beta \in \left [\beta_0 - \delta, \beta_0 + \delta \right ] \cap [\alpha, \infty),$ $\gamma \in \left [\gamma_0 - \delta, \gamma_0 + \delta \right ]$ and $\theta\in U$, let $L^{\beta, \theta, \gamma}_{\beta_0, \theta_0, \gamma_0, \varepsilon}$ be the Lip-norm on $V_{\theta} \oplus V_{\theta_0}$ as in the last theorem. For simplicity, we write $L$ for $L^{\beta, \theta, \gamma}_{\beta_0, \theta_0, \gamma_0, \varepsilon}.$ We show that $\mathrm{dist}_{\mathrm{GH}}^{\rho_{L}} \left (S \left (V_{\theta} \right ), S \left (V_{\theta_0} \right ) \right ) < \varepsilon.$ 

Let $\mu \in S \left (V_{\theta} \right ).$ By the choice of $\mathcal S,$ there is an $\Phi^{\mu} \in \mathcal S$ such that 
$$\rho_{L_{\beta, \gamma}} \left (\mu, \Phi^{\mu}_{\theta} \right ) \leq \left \|\mu - \Phi_{\theta}^{\mu} \right \|_{\ast}^{\prime} < \frac {\varepsilon} {2}.$$
We now show that $\rho_{L} \left (\mu, \Phi^{\mu}_{\theta_0} \right ) < \varepsilon,$ which shows that $S \left (V_{\theta} \right )$ is within the $\varepsilon$-neighbourhood of $S \left (V_{\psi} \right )$ with respect to $\rho_L.$

Let $u, v \in V,$ with $(u, v)$ viewed as an element of $V_{\theta} \oplus V_{\psi},$ and suppose that $L (u, v) \leq 1.$ Then $N (u, v) < \frac {\varepsilon} {2},$ so that 
$$\left \lvert \Phi^{\mu}_{\theta} (u) - \Phi^{\mu}_{\theta_0} (v) \right \rvert < \frac {\varepsilon} {2}.$$ Since this holds for all $u, v \in V$ with $L (u, v) \leq 1,$ it follows that 
$$\rho_{L} \left (\Phi^{\mu}_{\theta}, \Phi^{\mu}_{\theta_0} \right ) \leq \frac {\varepsilon} {2}.$$ 
Thus 
$$\rho_L \left (\mu, \Phi^{\mu}_{\theta_0} \right ) \leq \rho_L \left (\mu, \Phi^{\mu}_{\theta} \right ) + \rho_L \left (\Phi^{\mu}_{\theta}, \Phi^{\mu}_{\theta_0} \right ) < \varepsilon,$$ as claimed.

By reversing the roles of $\theta$ and $\theta_0,$ we see that $S \left (V_{\theta_0} \right )$ is also within the $\varepsilon$-neighbourhood of $S \left (V_{\theta} \right )$ with respect to $\rho_L,$ as desired.

\end{proof}

\begin{thm} \label{joint continuity qGH}
As before, let $(\beta_0,\gamma_0,\theta_0)$ be a fixed point in the parameter space. Then for a given $\varepsilon > 0$ there exists $\delta > 0$ and a neighborhood $U$ around $\theta_0$ such that for all $\theta \in U$, for all $\beta \geq \alpha$ with $\left \lvert \beta - \beta_0 \right \rvert < \delta$ and for all $\gamma > C$ with $\left \lvert \gamma - \gamma_0 \right \rvert < \delta,$ we have
$$\mathrm{dist}_{\mathrm {qGH}} \left (\left (A_{\theta}, L_{\ell}^{S,\beta, \gamma} \right ), \left (A_{\theta_0}, L_{\ell}^{S,\beta_0,\gamma_0} \right ) \right ) < \varepsilon.$$
\end{thm}

\begin{proof}
Let $\eta= \frac {\gamma_0 - C} {2}$ and $\varepsilon > 0$ be arbitrary. Then by virtue of Theorem \ref{tail threshold} and Theorem \ref{fin-dim threshold}, there exist $R > 0$ and $n_0 \in \mathbb N,$ such that for $\theta \in \Omega,$ for all $\beta \geq \alpha $ and $\gamma \geq \gamma_0 - \eta,$ we have
$$\mathrm{dist}_{\mathrm{qGH}} \left (\left (A_{\theta}, L^{S,\beta, \gamma}_{\ell} \right ), \left (A_{\theta}^{R}, L^{S,\beta,\gamma} \big \rvert_{A_{\theta}^{R}} \right ) \right ) < \frac {\varepsilon} {5},$$ and,
$$\mathrm{dist}_{\mathrm{qGH}} \left (\left (A_{\theta}^{R}, L^{S,\beta, \gamma}_{\ell} \big \rvert_{A_{\theta}^{R}} \right ), \left (A_{n_0, \theta}^{R}, L_{\ell}^{S,\beta,\gamma} \big \rvert_{A_{n_0, \theta}^{R}} \right )  \right ) < \frac {\varepsilon} {5}.$$
Thus by the triangle inequality we have
$$\mathrm{dist}_{\mathrm{qGH}} \left (\left (A_{\theta}, L_{\ell}^{S,\beta, \gamma} \right ), \left (A_{n_0, \theta}^{R}, L_{\ell}^{S,\beta, \gamma} \big \rvert_{A_{n_0, \theta}^{R}} \right ) \right ) < \frac {2 \varepsilon} {5},$$ for all $\theta \in \Omega,$ $\beta \geq\alpha $ and $\gamma \geq \gamma_0 - \eta.$ Thus, in particular, for $\theta = \theta_0,\beta=\beta_0$ and $\gamma = \gamma_0$ we have
$$\mathrm{dist}_{\mathrm{qGH}} \left (\left (A_{\theta_0}, L_{\ell}^{S,\beta_0, \gamma_0} \right ), \left (A_{n_0, \theta_0}^{R}, L_{\ell}^{S,\beta_0, \gamma_0} \big \rvert_{A_{n_0, \theta_0}^{R}} \right ) \right ) < \frac {2 \varepsilon} {5}.$$
Denote $\left (A_{n_0, \psi}^{R} \right )_{\mathrm {sa}}$ by $V_{\psi}$ for $\psi \in \Omega$ as before. In order to show the joint continuity, it is enough to get hold of some $0 < \delta < \delta_0$ such that for all $\theta \in U$, where $U$ is some open neighborhood of $\theta_0$, for all $\beta \geq \alpha$ with $\left \lvert \beta - \beta_0 \right \rvert < \delta$ and for all $\gamma > C$ with $\left \lvert \gamma - \gamma_0 \right \rvert < \delta,$ we have
$$\mathrm{dist}_{\mathrm{qGH}} \left (\left (A_{n_0, \theta}^{R}, L_{\ell}^{S,\beta, \gamma} \big \rvert_{A_{n_0, \theta}^{R}} \right ), \left (A_{n, \theta_0}^{R}, L_{\ell}^{S,\beta_0, \gamma_0} \big \rvert_{A_{n_0, \theta_0}^{R}} \right ) \right ) < \frac {\varepsilon} {5}.$$
But as 
\begin{displaymath}
    \mathrm{dist}_{\mathrm{qGH}} \left (\left (A_{n_0, \theta}^{R}, L_{\ell}^{S,\beta, \gamma} \big \rvert_{A_{n_0, \theta}^{R}} \right ), \left (A_{n_0, \theta_0}^{R}, L_{\ell}^{S,\beta_0, \gamma_0} \big \rvert_{A_{n_0, \theta_0}^{R}} \right ) \right )=\mathrm{dist}_{Q} \left (\left (V_{\theta}, L_{\beta, \gamma} \big \rvert_{V_{\theta}} \right ), \left (V_{\theta_0}, L_{\beta_0, \gamma_0} \big \rvert_{V_{\theta_0}} \right ) \right ),
\end{displaymath}
this is taken care of by virtue of Theorem \ref{joint continuity lem}. This completes the proof.

\end{proof}
As a consequence, we have the following theorem which is the main theorem of this subsection. We recall all the assumptions in the statement of the theorem for the reader's convenience:
\begin{thm}
    \label{mainthm} Let $(A,\rho,\Gamma)$ be a $C^{\ast}$-dynamical system such that\\
    (i) $A$ is a unital $C^{\ast}$-algebra with a Lip-norm $L_A$ making $(A,L_A)$ a compact quantum metric space with the finite dimensional approximation property as defined in Definition \ref{fdim approx prop}.\\
    (ii) $\rho$ is a {\bf Lip-isometric} action with respect to the Lip-norm $L_A$.\\
    (iii) $\Gamma$ is a {\bf discrete exact} group admitting a strongly continuous one-parameter family $\{\sigma_{\theta}\}_{\theta\in\Omega}$ of unitary $2$-cocycles where $\Omega$ is some compact parameter space.\\
    (iv) $\Gamma$ has a length function $\ell$ and the growth function $\lambda$ of $\Gamma$ satisfies $\lambda(n)\leq e^{Cn^{\alpha}}$ for some positive constant $C$ and some constant $0<\alpha\leq 1$.\\
    Then the three parameter family of compact quantum metric spaces $\{\big(A\rtimes_{r,\rho,\sigma_{\theta}}\Gamma,L^{S,\beta,\gamma}_{\ell}\big)\}_{\beta,\gamma,\theta}$ where $\beta\geq\alpha,\gamma>C$ and $\theta\in\Omega$ is jointly continuous with respect to the parameters $\beta,\gamma,\theta$. Here $L^{S,\beta.\gamma}_{\ell}$ is the $2$-parameter family of Lip-norms as defined in Section \ref{Lip-norm defn}.
\end{thm}

\subsection{The metric dimension} In this subsection, we are going to obtain bounds for the metric dimension of the family of CQMS introduced in the last subsection. We shall see that the bounds depend on the Lip-norm parameters $\beta,\gamma$ and are independent of the cocycle parameter $\theta$. The techniques to obtain bounds follow the techniques used in \cites{Soumalya-Arnab1,Soumalya-Arnab2}. We are going to retain all the notations used in the previous subsection. To obtain the bounds we do not need the discrete group $\Gamma$ to be exact.
\begin{thm}\label{Mdimupper}
Let $(A,\rho,\Gamma)$ be a $C^{\ast}$-dynamical system satisfying the standard assumptions; $\rho$ is a Lip-isometric action on $A$ with respect to the Lip-norm $L_A$ on $A$ so that we have a two parameter family of Lip-norms $\{L^{S,\beta,\gamma}_{\ell}\}_{\beta\geq \alpha,\gamma>C}$ on the cocycle twisted crossed product $A\rtimes_{r,\rho,\sigma}\Gamma$ where $\sigma$ is a unitary $2$-cocycle on $\Gamma$. Then 
\begin{displaymath}\mathrm {Mdim}_{L_{\ell}^{S, \beta, \gamma}} \left (A \rtimes_{r, \rho, \sigma} \Gamma \right ) \leq \begin{cases} \mu, \quad \mathrm{if}\ \beta > \alpha, \\ \frac {2 C + \mu (\gamma + C)} {\gamma - C}, \quad \mathrm{if}\ \beta = \alpha\ \mathrm{and}\ 0 < \alpha < 1, \\ \frac {C + \mu \gamma} {\gamma - C}, \quad \mathrm{if}\ \beta = \alpha = 1, \end{cases}\end{displaymath}
where $\mu = \mathrm{Mdim}_{L_A} (A).$

\end{thm}

\begin{proof}
First let us assume $0 < \alpha < 1.$ Fix some arbitrary $\delta>0$. In the following we denote the unit Lip-ball of $A$ by $\mathcal L_{1}$ i.e. $\mathcal{L}_1=\{a\in A:L_A(a)\leq 1\}$. Then for $\alpha \leq \beta < 1$ and $\gamma > C,$ for any $\mathfrak{a}=\sum\limits_{t\in\Gamma}\delta_ta_{t}\in C_{c}(\Gamma,\mathcal A,\sigma)$ such that $\mathfrak{a}\in( \mathcal{L}^{S,\beta,\gamma}_{\ell})_1$, applying the similar argument as in \eqref{tail estimation}, there exists a number $N$ such that $\left \|\sum\limits_{\ell(t)\geq N} \delta_t a_t\right\|_{\mathrm{red}}<\frac{\delta}{2}$. By virtue of Lemma \ref{threshold}, such a threshold $N$ is given by $$N = \left \lceil \left (\frac {2} {\eta} \log \left (\frac {2K} {\delta} \right ) \right )^{\frac {1} {\beta}} \right \rceil,$$ where $\eta = \gamma - C$ and $K = \frac{2}{\eta \beta} \left( \frac{2(1-\beta)}{e \eta \beta} \right)^{\frac{1-\beta}{\beta}}.$ Let $\lambda_N : = \left \lvert \left \{t \in \Gamma\ :\ \ell (t) \leq N \right \} \right \rvert.$ Get hold of a finite dimensional subspace $Y$ of $A$ such that $$\dim (Y) = D \left (\mathcal L_1, \frac {\delta} {2\lambda_N} \right ) \quad \mathrm{and}\ \quad \mathcal L_1 \subseteq_{\frac {\delta} {2\lambda_N}} Y.$$ Let $Y = \mathrm {span} \left \{y_1, \cdots, y_q \right \}.$ Define $$U_N : = \mathrm {span} \left \{\delta_t y_j\ :\ \ell (t) \leq N, 1 \leq j \leq q \right \}.$$ By definition of the Lip-norm $L^{S,\beta,\gamma}_{\ell}$, each $a_t$ is in $\mathcal L_1.$ So for each $a_t,$ get hold of some $y_t \in Y,$ such that $\left \|a_t - y_t \right \| < \frac {\delta} {2\lambda_N}.$ Then we have
\Bea
\left \|\sum\limits_{t \in \Gamma} \delta_t a_t - \sum\limits_{\ell(t)\leq N } \delta_t y_t \right \|_{\mathrm {red}} & \leq & \left \|\sum\limits_{t \in \Gamma} \delta_t a_t - \sum\limits_{\ell (t) \leq N} \delta_t a_t \right \|_{\mathrm {red}} + \left \|\sum\limits_{\ell (t) \leq N} \delta_t a_t - \sum\limits_{\ell (t) \leq N} \delta_t y_t \right \|_{\mathrm {red}} \\ & < & \frac {\delta} {2} + \sum\limits_{\ell (t) \leq N} \left \|a_t - y_t \right \| \\ & < & \frac {\delta} {2} + \frac {\delta} {2} \\ & = & \delta.  
\Eea
This shows that $$D \left (\mathcal L_1^{S, \beta, \gamma}, \delta \right ) \leq \lambda_N q = \lambda_N D \left (\mathcal L_1, \frac {\delta} {2\lambda_N} \right ).$$ Taking logarithms and dividing by $\log \delta^{-1} > 0$ yields
$$\frac {\log D \left (\mathcal L_1^{S, \beta, \gamma}, \delta \right )} {\log \delta^{-1}} \leq \frac {\log \lambda_N} {\log \delta^{-1}} + \frac {\log D \left (\mathcal L_1, \frac {\delta} {2\lambda_N} \right )} {\log \delta^{-1}}.$$
Let $\lambda_0 = \limsup\limits_{\delta \to 0^{+}} \frac {\log \lambda_N} {\log \delta^{-1}}.$ To evaluate the second term, set $x = \frac {\delta} {2\lambda_N}.$ As $\delta \to 0^{+},$ $N \to \infty,$ which forces $\lambda_N \to \infty$ and $x \to 0^{+}.$ We rewrite the second term as follows
$$\frac {\log D \left (\mathcal L_1, x \right )} {\log \delta^{-1}} = \frac {\log D \left (\mathcal L_1, x \right )} {\log x^{-1}} \left( 1 + \frac {\log 2 + \log \lambda_N} {\log \delta^{-1}} \right).$$
Taking the limit supremum as $\delta \to 0^{+},$ this converges to $\mu(1 + \lambda_0).$ Thus, we obtain
\begin{equation} \label{main inequality} \limsup\limits_{\delta \to 0^{+}} \frac {\log D \left (\mathcal L_1^{S, \beta, \gamma}, \delta \right )} {\log \delta^{-1}} \leq \lambda_0 + \mu(1 + \lambda_0) \end{equation}
To bound $\lambda_0,$ we use $\lambda_N \leq e^{C N^{\alpha}},$ which gives $\log \lambda_N \leq C N^{\alpha}.$ From the formula of threshold, $N < \left (\frac {2} {\gamma - C} (\log(2K) + \log \delta^{-1}) \right )^{\frac {1} {\beta}} + 1.$ Setting $t = \log \delta^{-1},$ we get
$$\lambda_0 \leq \limsup\limits_{t \to \infty} \frac {C \left[ \left (\frac {2} {\gamma - C} (t + \log(2K)) \right )^{\frac {1} {\beta}} + 1 \right]^{\alpha}} {t}.$$
Factoring out $t^{\frac {1} {\beta}}$ yields
$$\lambda_0 \leq C \limsup\limits_{t \to \infty} t^{\frac {\alpha} {\beta} - 1} \left[ \left (\frac {2} {\gamma - C} \left( 1 + \frac {\log(2K)} {t} \right) \right )^{\frac {1} {\beta}} + t^{-\frac {1} {\beta}} \right]^{\alpha}.$$
As $t \to \infty,$ the bracketed term converges to $\left( \frac {2} {\gamma - C} \right )^{\frac {\alpha} {\beta}}.$ Since $0 < \alpha \leq \beta < 1,$ we consider two cases for the exponent $\frac {\alpha} {\beta} - 1$:

\textbf{Case 1: $0 < \alpha < \beta$} \\
Here, $\frac {\alpha} {\beta} - 1 < 0,$ so that $t^{\frac {\alpha} {\beta} - 1} \to 0.$ Thus, $\lambda_0 = 0.$ Substituting into the \eqref{main inequality} yields
$$\limsup\limits_{\delta \to 0^{+}} \frac {\log D \left (\mathcal L_1^{S, \beta, \gamma}, \delta \right )} {\log \delta^{-1}} \leq \mu.$$

\textbf{Case 2: $\alpha = \beta$} \\
Here, $\frac {\alpha} {\beta} - 1 = 0,$ and hence the limit evaluates to $\lambda_0 \leq \frac {2C} {\gamma - C}.$ Substituting this into the \eqref{main inequality} yields
$$\limsup\limits_{\delta \to 0^{+}} \frac {\log D \left (\mathcal L_1^{S, \beta, \gamma}, \delta \right )} {\log \delta^{-1}} \leq \left( \frac {2C} {\gamma - C} \right) + \mu \left( 1 + \frac {2C} {\gamma - C} \right) = \frac {2C + \mu(\gamma + C)} {\gamma - C}.$$
The threshold $N$ for $\beta \geq 1$ is given by $$N = \left\lceil \left(\frac {1} {\eta} \log \left (\frac {1} {\eta \beta \varepsilon} \right) \right )^{\frac{1}{\beta}} \right \rceil.$$ Thus when $0 < \alpha < 1 \leq \beta$ or $1 = \alpha \leq \beta,$ using the same technique as above, we get $$\limsup\limits_{\delta \to 0^{+}} \frac {\log D \left (\mathcal L_1^{S, \beta, \gamma}, \delta \right )} {\log \delta^{-1}} \leq \mu,$$ when $\beta > \alpha,$ and, $$\limsup\limits_{\delta \to 0^{+}} \frac {\log D \left (\mathcal L_1^{S, \beta, \gamma}, \delta \right )} {\log \delta^{-1}} \leq \frac {C + \mu \gamma} {\gamma - C},$$ when $\beta = \alpha.$
Thus we have $$\mathrm {Mdim}_{L_{\ell}^{S, \beta, \gamma}} \left (A \rtimes_{r, \rho, \sigma} \Gamma \right ) \leq \begin{cases} \mu, \quad \mathrm{if}\ \beta > \alpha, \\ \frac {2 C + \mu (\gamma + C)} {\gamma - C}, \quad \mathrm{if}\ \beta = \alpha\ \mathrm{and}\ 0 < \alpha < 1, \\ \frac {C + \mu \gamma} {\gamma - C}, \quad \mathrm{if}\ \beta = \alpha = 1. \end{cases}$$
This completes the proof.

\end{proof}
Now to obtain lower bounds of metric dimension, we further assume that the growth function $\lambda$ of $\Gamma$ satisfies $e^{C^{\prime}n^{\alpha}}\leq\lambda(n)\leq e^{Cn^{\alpha}}$ for some positive constants $C,C^{\prime}$. This kind of condition is satisfied by large family of discrete groups of subexponential and exponential growths (\cite{Bridson-Int-Growth}).

\begin{thm}
Retaining all the notations of the previous theorem, we get the following lower bound of the metric dimension: 
\begin{displaymath}\mathrm{Mdim}_{L_{\ell}^{S, \beta, \gamma}} \left (A \rtimes_{r, \rho, \sigma} \Gamma \right ) \geq \begin{cases} \max \left \{\mathrm{Mdim}_{L_A} (A),\frac {C^{\prime}} {2 \gamma} \right \}, \quad \mathrm{if}\ \beta = \alpha, \\ \mathrm{Mdim}_{L_A} (A), \quad \mathrm{if}\ \beta > \alpha. \end{cases}\end{displaymath}

\end{thm}

\begin{proof}
Define a state $\tau$ on $B = A \rtimes_{r, \rho, \sigma} \Gamma$ by $\tau = \varphi \circ E$ where $\tau$ is some faithful state on $A$ and $E:A\rtimes_{r,\rho,\sigma}\Gamma\rightarrow A$ is the canonical conditional expectation. Let $\left (\pi_{\tau}, \mathcal H_{\tau}, \xi_{\tau} \right )$ be the associated GNS representation of $B.$ Then for $t, s \in \Gamma$ with $t \neq s$ we have
\Bea
\left \langle \pi_{\tau} \left (\delta_t 1 \right ) \xi_{\tau}, \pi_{\tau} \left (\delta_s 1 \right ) \xi_{\tau} \right \rangle & = & \tau \left (\left (\delta_t 1 \right )^{\ast} \left (\delta_s 1 \right ) \right ) \\ & = & \tau \left (\delta_{t^{-1} s} \overline {\sigma (t, t^{-1})} {\sigma (t^{-1}, s)} \right ) \\ & = & \begin{cases} \varphi (1),\ \quad \mathrm{if}\ t = s, \\ 0, \quad \mathrm{otherwise}. \end{cases} \\ & = & \begin{cases} 1, \quad \mathrm{if}\ t = s, \\ 0, \quad \mathrm{otherwise}. \end{cases}
\Eea
Consider the set $$U_{\delta} : = \left \{\delta_s 1\ :\ \ell (s) \leq \left (\frac {\log \delta^{-1}} {2 \gamma} \right )^{\frac {1} {\beta}} \right \}.$$ Then the above computation shows that the set $\pi_{\tau} \left (U_{\delta} \right ) \xi_{\tau}$ is a finite set of orthonomal vectors in $\mathcal H_{\tau}.$ Therefore by virtue of Voiculescu's lemma \cite{Voiculescu-Entp}*{Lemma 7.8} it follows that $$D \left (\pi_{\tau} \left (U_{\delta} \right ) \xi_{\tau}, \frac {1} {2} \right ) \geq \frac {3} {4} \left \lvert U_{\delta} \right \rvert \geq  \frac {3} {4} e^{C' {N_{\beta, \gamma, \delta}}^{\alpha}},$$ 
where $N_{\beta, \gamma, \delta} : = \left (\frac {\log \delta^{-1}} {2  \gamma} \right )^{\frac {1} {\beta}}.$ Also note that $\delta U_{\delta} \subseteq \mathcal L^{S, \beta, \gamma}_{1}.$ Thus we have
\Bea
D \left (\mathcal L^{S, \beta, \gamma}_{1}, \frac {\delta} {2} \right ) & \geq & D \left (\delta U_{\delta}, \frac {\delta} {2} \right ) \\ & = & D \left (U_{\delta}, \frac {1} {2} \right ) \\ & \geq & D \left (\pi_{\tau} \left (U_{\delta} \right ) \xi_{\tau}, \frac {1} {2} \right ) \\ & \geq & \frac {3} {4} e^{C' {N_{\beta, \gamma, \delta}}^{\alpha}}.
\Eea
Taking logarithm and dividing both sides by $\log \delta^{-1}$ yields
\Bea
\frac {\log D \left (\mathcal L^{S, \beta, \gamma}_{1}, \frac {\delta} {2} \right )} {\log \delta^{-1}} & \geq & \frac {\log \frac {3} {4}} {\log \delta^{-1}} + C' \frac {N_{\beta, \gamma, \delta}^{\alpha}} {\log \delta^{-1}} \\ & = & \frac {\log \frac {3} {4}} {\log \delta^{-1}} + \frac {C'} {\left (2 \gamma \right )^{\frac {\alpha} {\beta}}} \left (\log \delta^{-1} \right )^{\frac {\alpha} {\beta} - 1} 
\Eea
Finally, letting $\delta \to 0^{+},$ it follows that 
$$\mathrm {Mdim}_{L_{\ell}^{S, \beta, \gamma}} \left (A \rtimes_{r, \rho, \sigma} \Gamma \right ) = \limsup\limits_{\delta \to 0^{+}} \frac {\log D \left (\mathcal L^{S, \beta, \gamma}_{1}, \frac {\delta} {2} \right )} {\log \delta^{-1}} \geq \begin{cases} \frac {C'} {2 \gamma}, \quad \mathrm{if}\ \beta = \alpha, \\ 0, \quad \mathrm{if}\ \beta > \alpha. \end{cases}$$
Now note that for any $a \in \mathcal L_1,$ we have $\delta_e a \in \mathcal L^{S, \beta, \gamma}_{1}.$ Since $A$ is faithfully embedded in $A \rtimes_{r, \rho, \sigma} \Gamma$ as $\delta_e A,$ it follows that $$D \left (\mathcal L^{S, \beta, \gamma}_{1}, \delta \right ) \geq D \left (\mathcal L_1, \delta \right ),$$ for any $\delta > 0.$ Dividing both sides by $\log \delta^{-1}$ and letting $\delta \to 0^{+},$ it follows that
$$\mathrm{Mdim}_{L_{\ell}^{S, \beta, \gamma}} \left (A \rtimes_{r, \rho, \sigma} \Gamma \right ) \geq \mathrm{Mdim}_{L_A} (A).$$ Thus we have
$$\mathrm{Mdim}_{L_{\ell}^{S, \beta, \gamma}} \left (A \rtimes_{r, \rho, \sigma} \Gamma \right ) \geq \begin{cases} \max \left \{\mathrm{Mdim}_{L_A} (A),\frac {C'} {2 \gamma} \right \}, \quad \mathrm{if}\ \beta = \alpha, \\ \mathrm{Mdim}_{L_A} (A), \quad \mathrm{if}\ \beta > \alpha. \end{cases}$$
This completes the proof.
\end{proof}
\begin{cor}\label{Mdim_greaterthancritical}
    Retaining all the notations as before, for any parameter $\beta>\alpha$,
    \begin{displaymath}
    \mathrm{Mdim}_{L_{\ell}^{S, \beta, \gamma}} \left (A \rtimes_{r, \rho, \sigma} \Gamma \right )=\mathrm{Mdim}_{L_A}(A)    
    \end{displaymath}
    for all $\gamma>C$. In particular, when the base $C^{\ast}$-algebra $A$ is $C(X)$ for some compact metric space $(X,d)$, for all $\beta>\alpha$ and $\gamma>C$, 
     \begin{displaymath}
    \mathrm{Mdim}_{L_{\ell}^{S, \beta, \gamma}} \left (A \rtimes_{r, \rho, \sigma} \Gamma \right )=\mathrm{Kol}(X,d),    
    \end{displaymath}
    where $\mathrm{Kol}(X,d)$ denotes the Kolmogorov dimension of $(X,d)$.
\end{cor}
When the base algebra $A$ is $C(X)$ for some classical compact metric space then adapting the arguments in \cite{Soumalya-Arnab2}*{Theorem 3.11}, one can improve the lower bound of the metric dimension at the critical parameter $\beta=\alpha$. We state the improved lower bound without proof.

\begin{lem}\label{improvedlowerbound}
    With all the notations as used earlier, when $A$ is $C(X)$ for some classical, compact metric space $(X,d)$, then for $\beta=\alpha$ and all $\gamma>C$,
    \begin{displaymath}
    \mathrm{Mdim}_{L_{\ell}^{S, \beta, \gamma}} \left (A \rtimes_{r, \rho, \sigma} \Gamma \right )\geq \mathrm{Kol}(X,d)+\frac{C^{\prime}}{\gamma}.\end{displaymath}
\end{lem}

\begin{rem}\label{lowersemicontinuity}
    When the group $\Gamma$ is exact and the base algebra $A$ is $C(X)$ for some compact metric space $(X,d)$, Lemma \ref{improvedlowerbound} and Corollary \ref{Mdim_greaterthancritical} together with the Theorem \ref{mainthm} demonstrates the failure of lower semicontinuity of the metric dimension function with respect to the quantum Gromov-Hausdorff distance as $\beta\downarrow\alpha$. 
\end{rem}

\begin{rem}\label{qGHcollapse}
Let $\mu = \mathrm {Kol} (X, d).$ Consider a sequence $\left \{\gamma_n \right \}_{n \geq 1}$ with $\gamma_n > C$ for all $n \geq 1$ in such a way that $\gamma_{n + 1} > \frac {2 C \gamma_n (1 + \mu)} {C'} + C$ for all $n \geq 1.$ Then $\frac {C + \mu \gamma_{n + 1}} {\gamma_{n + 1} - C } < \frac {2 C + \mu \left (\gamma_{n + 1} + C \right )} {\gamma_{n + 1} - C} < \mu + \frac {C'} {\gamma_n}$ for all $n \geq 1.$ Thus by virtue of Lemma \ref{improvedlowerbound}, we can get hold of a countable collection of CQMS structure on $A \rtimes_{r, \rho, \sigma} \Gamma$ in terms of the Lip-norms $L_{\ell}^{S, \beta, \gamma_n}$ at the critical parameter $\beta = \alpha$ (for any $\alpha > 0$) such that the associated metric dimensions form a strictly decreasing sequence. Thus, in light of Theorem \ref{qGH vs Mdim}, the quantum Gromov-Hausdorff distance between any pair of distinct CQMSs from the countable collection is strictly positive. This prevents the continuity and convergenvce results established in Theorem \ref{joint continuity qGH} and Theorem \ref{convergence qGH} from reducing to complete triviality. 
\end{rem}

\section{Examples} \label{Examples}
Let $\Gamma_{g}$ be the fundamental group of a closed orientable surface of genus $g\geq 2$. Then recall that $\Gamma_g$ is a finitely generated discrete group with the following standard presentation:
\begin{displaymath}
    \Gamma_g=\langle a_{1},b_{1},a_2,b_2,\ldots,a_g,b_g|[a_{1},b_1][a_2,b_2]\ldots[a_g,b_g]=1\rangle.
\end{displaymath}
The following properties of $\Gamma_g$ are well known:\\
(i) $\Gamma_g$ has exponential growth (\cite{Sambu-tight}).\\
(ii) $\Gamma_g$ is exact (\cite{Anan-Del-Exact}).\\
$\Gamma_g$ has a one-parameter family of continuous unitary $2$-cocycles. in fact, $H^2(\Gamma_g,U(1))\cong\mathbb{R}/\mathbb{Z}.$ This follows from \textbf{Universal Coefficient Theorem} \cite{Hatcher-Alg-Top}*{Theorem 3.2} as, by \cite{Brown-Cohom}*{Example 3, Section 1, Chapter III}, $$H^{\ast} \left (\Gamma_g, U (1) \right ) \cong H^{\ast} \left (K \left (\Gamma_g, 1 \right ), U (1) \right ),$$ and $K \left (\Gamma_g, 1 \right ) \cong \Sigma_g.$ The isomorphism id essentially obtained by exponentiating a generator of $H^{2}(\Gamma_g,\mathbb{Z})\cong\mathbb{Z}.$

\subsection{Action on the profinite completion} 
It is well known that $\Gamma_g$ is residually finite (see page no. 105 of \cite{Wikes-Profinite} for example). Let us consider the profinite completion $\widehat{\Gamma_g}$ of $\Gamma_g$. We recall the definition of profinite completion. For details, the reader is referred to \cite{Wikes-Profinite}*{Definition 1.2.5}. Consider the inverse system of finite index normal subgroups of $\Gamma_g$ under the set theoretic inclusion. For any $N_1\subseteq N_2$, consider the natural map $\phi_{N_1 N_2}:\Gamma_g/N_1\rightarrow \Gamma_g/N_2$. Then the profinite completion is the inverse limit of this system. We have the canonical projection maps $p_{N}:\widehat{\Gamma_g}\rightarrow \Gamma_g/N$ for all finite index normal subgroups $N\subset\Gamma_g$. As $\Gamma_g$ is finitely generated, one can obtain the profinite completion $\widehat{\Gamma_g}$ as the inverse limit of countably many finite subgroups $\Gamma/N$. In particular, one can choose a particular inverse system given by $\Gamma_g \supseteq N_{1}\supseteq N_{2} \supseteq \ldots$ to obtain the inverse limit. This is done by using the fact that a finitely generated discrete group has countably many normal subgroups of finite index. Therefore, without loss of generality, we can take the transition maps $\phi_{ij}:\Gamma/N_{i}\rightarrow\Gamma/N_{j}$. Then any element of $\widehat{\Gamma_g}$ can be written as $(sN_1,sN_2,sN_3,\ldots)$ for $s\in\Gamma_g$. It is well known that $\widehat{\Gamma_g}$ is a compact, Hausdorff totally disconnected topological space. Therefore, $C(\widehat{\Gamma_g})$ is an AF-algebra. More precisely $C(\widehat{\Gamma_g})=\lim\limits_{j}(C(\Gamma/N_j),p_{N_j}^{\ast})$ where $p_{N_j}^{\ast}:C(\Gamma/N_j)\rightarrow C(\widehat{\Gamma_g})$ is the natural injective $C^{\ast}$-homomorphism identifying $C(\Gamma/N_j)$ as a $C^{\ast}$-subalgebra of $C(\widehat{\Gamma_g})$ for all $j$. The group $\Gamma_g$ naturally acts on $\widehat{\Gamma_g}$ by left translation: $t.((sN_1,sN_2,sN_3,\ldots)):=(tsN_1,tsN_2,tsN_3,\ldots)$. Consequently $\Gamma_g$ acts on the AF-algbera $C(\widehat{\Gamma_{g}})$. The action is given by $\rho_t(f)(x):=f(t^{-1}.x)$ for $t\in \Gamma_g$, $f\in C(\widehat{\Gamma_g})$ and $x\in\widehat{\Gamma_g}$. It is straightforward to verify that the canonical projection group homomorphisms $p_{N_j}:\widehat{\Gamma_g}\rightarrow \Gamma_g/N_j$ are equivariant with respect to the left translation actions. Using this, it is easy to see that the left translation action preserves the finite dimensional subalgebras $A_{j}:=p_{N_j}^{\ast}(C(\Gamma/N_j))$. 
\begin{lem}
    Let $\Gamma_g$ be a surface group as above; $\mu$ be the unique normalized Haar measure on the compact group $\widehat{\Gamma_g}$. There are natural conditional expectations $E_{j}:C(\widehat{\Gamma_g})\rightarrow A_{j}$ for all $j$ such that $E_{j}$ preserves the faithful tracial state $\tau$ corresponding to the Haar measure $\mu$. Moreover, the left translation action of $\Gamma_g$ on $C(\widehat{\Gamma_g})$ commutes with $E_j$ for all $j$.
\end{lem}
\begin{proof}
 Note that $\Gamma_g/N_{j}$ is isomorphic to $\widehat{\Gamma_g}/K_{j}$ as topological groups where $K_{j}:=\mathrm{Ker}(p_{N_j})$ is a compact subgroup of $\widehat{\Gamma_g}$. Then any $f\in A_j$ is canonically identified with $f\in C(\widehat{\Gamma_g})$ which are right $K_{j}$ invariant i.e. $A_{j}=\{f\in C(\widehat{\Gamma_g}):f(x)=f(xk) \ for \ k\in K_j\}$. Fix the unique normalized Haar measure $\mu_j$ on $K_{j}$ and define $E_j:C(\widehat{\Gamma_g})\rightarrow A_{j}$ by
 \begin{displaymath}
     E_j(f)(x):=\int\limits_{K_j}f(xk)d\mu_j(k).
 \end{displaymath}
 Then it is routine to verify that $E_j$'s are conditional expectations onto $A_j$ for all $j$. Also 
 \Bea
 \tau(E_j(f))&=&\int_{\widehat{\Gamma_g}}E_{j}(f)(x)d\mu(x)\\
 &=& \int_{\widehat{\Gamma_g}}\Big(\int_{K_j}f(xk)d\mu_{j}(k)\Big)d\mu\
\Eea
Now interchanging the integral and using the right invariance of the Haar measure $\mu$ on $\widehat{\Gamma_g}$, it is easy to see that $\tau(E_j(f))=\tau(f)$ for all $f\in C(\widehat{\Gamma_g})$. It is also easy to verify that $\rho_t$ commutes with $E_{j}$ for all $j$.
 
\end{proof}
Thanks to the above lemma, by \cite{Agui-Latre-AF}*{Theorem 3.5}, we have a CQMS structure on $(C(\widehat{\Gamma_g}),L)$ where $L$ is the Lip-norm given by the following formula for a chosen sequence of positive real numbers $\{\beta_j\}_{j\in\mathbb{N}}$ decreasing to $0$:
\begin{equation}\label{lip1}
   L(f):=\sup\limits_j\frac{\left \| f-E_{j}(f) \right \|}{\beta_j},  
\end{equation}
where $\|\cdot\|$ is the sup-norm on $C(\widehat{\Gamma_g})$. This Lip-norm satisfies the conditions of Definition \ref{fdim approx prop}. As the left translation action of $\Gamma_g$ on $C(\widehat{\Gamma_g})$ commutes with all $E_{j}$'s, a straightforward argument gives us the following lemma:
\begin{lem}
For any $t\in\Gamma_g$, $\rho_t$ is isometric with respect to the Lip-norm $L$ on $C(\widehat{\Gamma_g})$ i.e. for any $f\in C(\widehat{\Gamma_g})$,
    \begin{displaymath}
        L(\rho_{t}(f))=L(f), \ \forall \ f\in C(\widehat{\Gamma_g}),
    \end{displaymath}
    where $\rho_{t}(f)(x)=f(t^{-1}.x)$ as before.
\end{lem}
Therefore, combining all these we have the following:
\begin{prop}
    The $C^{\ast}$-dynamical system $\Big(C(\widehat{\Gamma_g}),\rho,\Gamma_g)$ satisfies all the conditions of Theorem \ref{mainthm} where $C(\widehat{\Gamma_g})$ carries the CQMS structure from the Lip-norm given by Formula \ref{lip1}. 
\end{prop}
\subsection{Action on \texorpdfstring{$SU(2)$}{SU(2)}} Consider the compact Lie group $SU(2)$ and a group homomorphism $\zeta:\Gamma_g\rightarrow SU(2)$. Any such choice is going to give a $C^{\ast}$-dynamical system $\big(C(SU(2)),\rho,\Gamma_g\big)$ where $\rho_{t}(f)(x):=f(\zeta(t)^{-1}.x)$ for $t\in \Gamma_g$ and $x\in SU(2)$. Also the left translation action of $SU(2)$ on itself produces an ergodic action $\chi$ on $C(SU(2))$. Then it is easy to see that $\rho_{t}\in\mathrm{Aut}\big(C(SU(2))\big)$ is nothing but $\chi_{\zeta(t)}$ for all $t\in\Gamma_g$. Choosing an adjoint invariant length function $\ell$ on $SU(2)$ (i.e. $\ell(x)=\ell(yxy^{-1})$ for all $x,y\in SU(2)$) produces a Lip-norm $L$ on $C(SU(2))$ given by
\begin{equation}\label{lip2}
    L(f):=\sup_{x\neq e}\frac{\|\chi_{x}(f)-f \|}{\ell(x)}.
\end{equation}
\begin{lem}
    The action $\rho$ is Lip-isometric i.e. for any $t\in\Gamma_g$, $L(\rho_t(f))=L(f)$ for all $f\in C(SU(2))$.
\end{lem}
\begin{proof}
For any $t\in\Gamma_g$, 
\Bea
L(\rho_t(f))&=& \sup_{x\neq e}\frac{\|\chi_{x}(\rho_t(f))-\rho_t(f) \|}{\ell(x)}\\
&=& \sup_{x\neq e}\frac{\|\chi_{x}(\chi_{\zeta(t)}(f))-\chi_{\zeta(t)}(f) \|}{\ell(x)}\\
&=&\sup_{x\neq e}\frac{\|\chi_{\zeta(t)}\big(\chi_{\zeta(t)^{-1}x\zeta(t)}(f)-f\big) \|}{\ell(x)}
\Eea
Using the facts that $\chi_{\zeta(t)}$ is norm isometric and $\ell$ is adjoint invariant we get that the last expression is equal to $L(f)$.\end{proof}
Now it is well known that the Lie group $SU(2)$ admits a faithful, finite dimensional unitary representation. Therefore, by Lemma 8.3 and Lemma 8.4 of \cite{Rieffel-Gromov}, we see that there are finite dimensional spaces $B_n\subset C(SU(2))$, a sequence of maps $P_{n}:C(SU(2))\rightarrow B_{n}$ and a sequence of real numbers $\delta_{n}$ going to zero such that $\lvert\lvert f-P_{n}(f)\rvert\rvert\leq\delta_n L(f)$, $\lvert\lvert P_n\rvert\rvert\leq 1$ and $L(P_n(f))\leq L(f)$ for all $f\in SU(2)$.  Combining all these we conclude that
\begin{prop}
    The $C^*$-dynamical system $\big(C(SU(2)),\rho,\Gamma_g\big)$ satisfies all the hypotheses of Theorem \ref{mainthm} where $C(SU(2))$ carries the CQMS structure from the Lip-norm given by Formula \ref{lip2}.
\end{prop}

\appendix
\section{Some Useful Estimations}

\begin{lem} \label{threshold}
Let $\alpha, \beta, \gamma,$ and $C$ be real constants satisfying $\beta \ge \alpha > 0$ and $\gamma > C > 0$, and define $\eta = \gamma - C$. For any $\varepsilon > 0$, there exists an integer $N \ge 1$ such that for all integers $n \ge N$,
\begin{equation*}
    \sum_{k=n+1}^{\infty} e^{-\gamma k^\beta + C k^\alpha} \le \varepsilon.
\end{equation*}
Furthermore, an explicit such threshold $N$ is given by
\begin{equation*}
    N = \begin{cases}
    \left\lceil \left(\frac {1} {\eta} \log \left (\frac {1} {\eta \beta \varepsilon} \right) \right )^{\frac{1}{\beta}} \right \rceil, & \mathrm{if }\ \beta \ge 1, \\[12pt]
    \left\lceil \left( \frac{2}{\eta} \log \left (\frac {K} {\varepsilon} \right) \right)^{\frac{1}{\beta}} \right\rceil, & \mathrm{if }\ 0 < \beta < 1,
    \end{cases}
\end{equation*}
where $K = \frac{2}{\eta \beta} \left( \frac{2(1-\beta)}{e \eta \beta} \right)^{\frac{1-\beta}{\beta}}.$
\end{lem}

\begin{proof}
First, we establish an upper bound for tail of the infinite sum using the integral test. For all $k \ge n + 1 \ge 1$, since $\beta \ge \alpha > 0,$ we have $k^{\beta} \geq k^{\alpha}.$ Since $C > 0,$ multiplying by $C$ yields $C k^{\alpha} \geq C k^{\beta}.$ Substituting this into the exponent of the summand, we obtain
\begin{equation*}
    -\gamma k^\beta + C k^\alpha \le -\gamma k^\beta + C k^\beta = -(\gamma - C) k^\beta.
\end{equation*}
Let $\eta = \gamma - C.$ Since $\gamma > C,$ we have $\eta > 0.$ This shows that the $k$-th term in the summation is bounded above by $e^{-\eta k^\beta}.$ The function $f(x) = e^{-\eta x^\beta}$ is positive and strictly decreasing on the interval $[1, \infty).$ Therefore, applying integral test we have
\begin{equation}
    S(n) \le \sum_{k=n+1}^{\infty} e^{-\eta k^\beta} \le \int_{n}^{\infty} e^{-\eta x^\beta} \, dx.
\end{equation}
In order to bound the integral we will consider the following two cases $:$

\vspace{2mm}

\textbf{Case 1:} Assume $\beta \geq 1.$ Then for all $x \ge n \ge 1$ it follows that $(x/n)^{\beta - 1} \ge 1.$ Thus we have
\Bea
    S(n) & \leq & \int_{n}^{\infty} \left( \frac{x}{n} \right)^{\beta - 1} e^{-\eta x^\beta} \, dx \\
    & = & \frac{1}{n^{\beta - 1}} \left[ -\frac{1}{\eta \beta} e^{-\eta x^\beta} \right]_n^\infty \\
    & = & \frac{e^{-\eta n^\beta}}{\eta \beta n^{\beta - 1}}.
\Eea
Since $n \geq 1$ and $\beta \ge 1$, we have $n^{\beta - 1} \ge 1,$ which allows us to simplify the bound to $S(n) \leq \frac{e^{-\eta n^\beta}}{\eta \beta}.$ Thus in order to ensure $S(n) \leq \varepsilon,$ it is sufficient to require
\begin{equation*}
    \frac{e^{-\eta n^\beta}}{\eta \beta} \le \varepsilon \implies e^{-\eta n^\beta} \le \eta \beta \varepsilon.
\end{equation*}
Taking logarithm in both sides and isolating $n$ yields the explicit threshold for $n.$

\textbf{Case 2:} Assume $0 < \beta < 1$. We apply the substitution $u = x^\beta$, which gives $dx = \frac{1}{\beta} u^{\frac{1-\beta}{\beta}} \, du.$ Then the integral bound becomes
\Bea
    S(n) & \leq & \frac{1}{\beta} \int_{n^\beta}^{\infty} u^{\frac{1-\beta}{\beta}} e^{-\eta u} \, du \\ & = & \frac{1}{\beta} \int_{n^\beta}^{\infty} \left( u^{\frac{1-\beta}{\beta}} e^{-\frac{\eta}{2} u} \right) e^{-\frac{\eta}{2} u} \, du.
\Eea
Let $p = \frac{1-\beta}{\beta} > 0$. We would like to have the global maximum of the function $g(u) = u^p e^{-\frac{\eta}{2} u}$ on $[0, \infty).$ It is easy to see that $g$ is differentiable and the equation $g' (u) = 0$ yields a unique critical point at $u = \frac{2p}{\eta}.$ Since $g$ is increasing on $\left [0, \frac {2 p} {\eta} \right ]$ and decraesing on $\left [\frac {2 p} {\eta}, \infty \right ),$ it follows that $g$ has the global maximum at $u = \frac {2 p} {\eta}.$  Evaluating $g$ at this maximum gives
\begin{equation*}
    M = g\left(\frac{2p}{\eta}\right) = \left( \frac{2(1-\beta)}{e \eta \beta} \right)^{\frac{1-\beta}{\beta}}.
\end{equation*}
Since $g(u) \le M$ for all $u \ge 0,$ it follows that
\Bea
    S(n) & \leq & \frac{M}{\beta} \int_{n^\beta}^{\infty} e^{-\frac{\eta}{2} u} \, du \\
    & = & \frac{M}{\beta} \left[ -\frac{2}{\eta} e^{-\frac{\eta}{2} u} \right]_{n^\beta}^\infty \\
    & = & \frac{2M}{\eta \beta} e^{-\frac{\eta} {2} n^{\beta}} \\ & = & K e^{-\frac {\eta} {2} n^{\beta}},
\Eea
where $K = \frac{2M}{\eta \beta}.$ Thus in order to guarantee $S(n) \leq \varepsilon,$ we require
\begin{equation*}
    K e^{-\frac{\eta}{2} n^\beta} \le \varepsilon \implies e^{-\frac{\eta}{2} n^{\beta}} \le \frac{\varepsilon}{K}.
\end{equation*}
Taking logarithm in both sides we have the explicit threshold for $n.$ This completes the proof.
\end{proof}
\begin{rem} \label{gamma independence}
Note that the threshold is a decreasing function in both $\beta$ and $\gamma$. Therefore, given a fixed $\epsilon>0$, if a threshold $N$ works for some $\beta=\beta_0$ and $\gamma=\gamma_0$, then the same threshold would work for all $\beta\geq\beta_0$ and all $\gamma\geq \gamma_0$. 
\end{rem}
\begin{lem} \label{finiteness}
Let $\alpha, \beta, C,$ and $\gamma_0$ be real numbers such that $\alpha > 0,$ $\beta \geq \alpha,$ and $\gamma_0 > C > 0.$ Let $S(\gamma)$ be the series defined by
$$S(\gamma) = \sum\limits_{n = 1}^{\infty} e^{-\gamma (n - 1)^{\beta} + C n^{\alpha}}$$
Then $S(\gamma)$ is uniformly bounded on the interval $[\gamma_0, \infty).$ That is, there exists a constant $M < \infty$ (depending only on $\alpha, \beta, C,$ and $\gamma_0$) such that $S(\gamma) \leq M$ for all $\gamma \geq \gamma_0.$
\end{lem}

\begin{proof} Let $\gamma \in [\gamma_0, \infty).$ For all integers $n \geq 1,$ we have $(n-1)^\beta \geq 0.$ Consequently,
$$-\gamma (n-1)^\beta \leq -\gamma_0 (n-1)^\beta$$
Since the exponential function is strictly monotonically increasing on $\mathbb{R},$ it follows that
$$e^{-\gamma (n - 1)^{\beta} + C n^{\alpha}} \leq e^{-\gamma_0 (n - 1)^{\beta} + C n^{\alpha}}$$
Summing over all $n \in \mathbb{N},$ we obtain
$$S(\gamma) \leq S(\gamma_0) = \sum_{n = 1}^{\infty} e^{-\gamma_0 (n - 1)^{\beta} + C n^{\alpha}}$$
In order to prove that $S(\gamma)$ is uniformly bounded on $[\gamma_0, \infty),$ it is therefore enough to show that $S(\gamma_0)$ converges to a finite value $M.$

Let $K_n = C n^{\alpha} - \gamma_0 (n - 1)^{\beta} = n^\alpha \left( C - \gamma_0 n^{\beta - \alpha} \left(1 - \frac{1}{n}\right)^\beta \right).$ Note that, if $\beta = \alpha,$ then $K_n = n^\alpha \left( C - \gamma_0 \left(1 - \frac{1}{n}\right)^\alpha \right).$
Since $\lim\limits_{n \to \infty} \left(1 - \frac{1}{n}\right)^\alpha = 1,$ the term inside the parenthesis converges to $C - \gamma_0 < 0.$
Let $\delta_1 = \frac{\gamma_0 - C}{2} > 0.$ Then there exists an integer $N_1$ such that for all $n \geq N_1,$
$$C - \gamma_0 \left(1 - \frac{1}{n}\right)^\alpha \leq -\delta_1$$
Therefore, for all $n \geq N_1,$ we have $K_n \leq -\delta_1 n^\alpha.$
For $\beta > \alpha,$ we have $\lim\limits_{n \to \infty} n^{\beta - \alpha} \left(1 - \frac{1}{n}\right)^{\beta} = \infty.$
Thus, the expression $\left( C - \gamma_0 n^{\beta - \alpha} \left(1 - \frac{1}{n}\right)^\beta \right)$ diverges to $-\infty.$
Fix any arbitrary $\delta_2 > 0.$ Then there exists an integer $N_2$ such that for all $n \geq N_2,$
$$C - \gamma_0 n^{\beta - \alpha} \left(1 - \frac{1}{n}\right)^{\beta} \leq -\delta_2$$
Therefore, for all $n \geq N_2,$ we have $K_n \leq -\delta_2 n^{\alpha}.$

\vspace{0.5em}
In either case, there exists a $\delta > 0$ and an integer $N \in \mathbb{N}$ (depending upon $\delta$) such that for all $n \geq N,$ 
$$e^{K_n} \leq e^{-\delta n^\alpha}$$
In order to show that $S \left (\gamma_0 \right )$ is finite, it is enough to show that the tail of the infinite sum representing $S \left (\gamma_0 \right )$ is finite. We now bound the tail of $S \left (\gamma_0 \right ),$ starting from index $N.$
$$\sum_{n=N}^{\infty} e^{K_n} \leq \sum_{n=N}^{\infty} e^{-\delta n^\alpha}$$
But the infinite sum on the right hand side can be shown to converge excatly in the similar fashion as in the proof of Lemma \ref{threshold}. 
\end{proof}

$\mathbf{Data\ Availability}:$ Not Applicable.

\vspace{2em}

$\mathbf{Conflict\ of\ Interest}:$ The authors declare that there is no conflict of interest.

\bibliographystyle{amsplain}
\bibliography{References}

\end{document}